\documentclass[preprint,12pt]{elsarticle}

\usepackage{amssymb}
\usepackage{amsmath}
\usepackage{amsthm}
\usepackage{mathrsfs}  

\usepackage{subcaption}
\usepackage{enumitem}

\usepackage[symbol]{footmisc}

\DeclareMathOperator*{\argmin}{arg\,min}
\DeclareMathOperator*{\rank}{rank}
\DeclareMathOperator*{\sign}{sign}

\theoremstyle{remark}
  \newtheorem*{remark}{Remarks}

\newtheorem{theorem}{Theorem}[section]
\newtheorem{lemma}{Lemma}[section]

\newtheorem{definition}{Definition}[section]

\newcommand{\eq}[1]{(\ref{eq:#1})}
\newcommand{\lem}[1]{Lemma~\ref{lemma:#1}}
\newcommand{\thm}[1]{Theorem~\ref{theorem:#1}}
\newcommand{\sect}[1]{Section~\ref{sec:#1}}
\newcommand{\fig}[1]{Figure~\ref{fig:#1}}

\begin{document}

\begin{frontmatter}



\title{On the optimal rank-1 approximation of matrices in the Chebyshev norm\footnote[1]{This work was supported by the Moscow Center of Fundamental and Applied Mathematics at INM RAS (Agreement with the Ministry of Education and Science of the Russian Federation No.075-15-2022-286).}}

\author[inst1]{Stanislav Morozov\footnote[2]{E-mail addresses: \texttt{stanis-morozov@yandex.ru}, \texttt{matsmir98@gmail.com}, \texttt{nikolai.zamarashkin@gmail.com}}}

\affiliation[inst1]{organization={Marchuk Institute of Numerical Mathematics of Russian Academy of Sciences},%Department and Organization
            addressline={Gubkin Street 8}, 
            city={Moscow},
            postcode={119333}, 
            % state={State One},
            country={Russia}}

\author[inst1,inst2]{Matvey Smirnov}
\author[inst1,inst2]{Nikolai Zamarashkin}

\affiliation[inst2]{organization={Lomonosov Moscow State University},%Department and Organization
            addressline={Leninskie Gory 1}, 
            city={Moscow},
            postcode={119991}, 
            country={Russia}}
\begin{abstract}
The problem of low rank approximation is ubiquitous in science. Traditionally this problem is solved in unitary invariant norms such as Frobenius or spectral norm due to existence of efficient methods for building approximations. However, recent results reveal the potential of low rank approximations in Chebyshev norm, which naturally arises in many applications. In this paper we tackle the problem of building optimal rank-1 approximations in the Chebyshev norm. We investigate the properties of alternating minimization algorithm for building the low rank approximations and demonstrate how to use it to construct optimal rank-1 approximation. As a result we propose an algorithm that is capable of building optimal rank-1 approximations in Chebyshev norm for moderate matrices.
\end{abstract}


% \begin{highlights}
% \end{highlights}


\end{frontmatter}

% \linenumbers

\section{Introduction}
\label{sec:intro}
To date, the problem of low-rank approximation of matrices is a critical component in many areas of science, such as computational mathematics \cite{bebendorf2008means}, computational fluid dynamics \cite{son2014data}, recommender systems \cite{he2016fast}, machine learning \cite{yang2018oboe}, and others. This problem can be easily solved in unitary invariant norms for example via SVD (singular value decomposition) or ALS (alternating least squares) algorithm. However, in many applications the elementwise approximations (i.e. in Chebyshev norm) are more appropriate. %For example, it is often necessary to compress a matrix, preserving the values of its elements in the best way. This task immediately leads to the problem of approximating the matrix in the Chebyshev norm.
%Moreover, unitary invariant norms have some theoretical limitations. The quality of approximation depends on the rate of fall of singular values. At the same time recent results \cite{udell2019big} show that low-rank approximations in the Chebyshev norm does not require the singular values to decrease rapidly. Also in \cite{udell2019big} authors demonstrate that matrices arising in machine learning models admit good approximations by low-rank matrices in the Chebyshev norm. This observation explains the success of low-rank representations in data analysis.

Formally the problem of rank-$r$ Chebyshev approximation is formulated as follows. Given matrix $A \in \mathbb{R}^{m \times n}$ and integer $r$, it is required to find $\widehat{U} \in \mathbb{R}^{m \times r}$ and $\widehat{V} \in \mathbb{R}^{n \times r}$ such that
\begin{equation*}
    \|A - \widehat{U} \widehat{V}^T\|_C = \inf\limits_{U \in \mathbb{R}^{m \times r}, V \in \mathbb{R}^{n \times r}} \|A - U V^T\|_C,
\end{equation*}
where $\|X\|_C = \max\limits_{i,j} |x_{ij}|$ is the Chebyshev norm.

%The problem of low-rank matrix approximations in the Chebyshev norm is little addressed in literature. As far as we know there are no algorithms for building \textit{optimal} approximations in Chebyshev norm. In \cite{udell2019big} the authors show that with a fixed spectral norm, the $\varepsilon$-rank of the matrix increases logarithmically with the growth of the matrix size. However, the method proposed in \cite{udell2019big} turns out to be impractical for building such approximations. In \cite{zamarashkin2022best} the authors propose a method that makes it possible to efficiently construct low-rank approximations in the Chebyshev norm, and the accuracy of the resulting approximations is no worse than those guaranteed in \cite{udell2019big}. Another interesting work on the Chebyshev approximations is \cite{daugavet1971uniform}, where the author uses alternating minimization algorithm for build local optima in the problem of rank-1 approximation.

In this paper we address the problem of rank-1 approximations in the Chebyshev norm. More precisely, we investigate the \textit{alternating minimization algorithm} and its structure. The algorithm can be briefly summarized as follows: fix arbitrary $v^{(0)}$ and generate sequences $v^{(k)}$ and $u^{(k)}$ such that $$\|A - u^{(k)}(v^{(k-1)})^T\|_C = \min\limits_{u}\|A - u(v^{(k-1)})^T\|_C,$$ $$\|A - u^{(k)}(v^{(k)})^T\|_C = \min\limits_{v}\|A - u^{(k)}v^T\|_C$$ for all $k \in \mathbb N$. We show that for almost all matrices $A$ (in the sense of Lebesgue measure) it is possible to choose the solutions $u^{(k)}$ and $v^{(k)}$ of the foregoing minimization problems in a concrete way, if all the components of $v^{(0)}$ are non-zero. Moreover, the limit of errors
\begin{equation*}
    \|A - u^{(1)}(v^{(0)})^T\|_C, \|A - u^{(1)}(v^{(1)})^T\|_C, \dots
\end{equation*}
depends only on the signs of components of $v^{(0)}$. Finally, it is always possible to extract a converging subsequence from the sequence of matrices $u^{(k)} (v^{(k)})^T$ and its limit gives one of the best rank-1 approximations for $A$ from a relatively big set of pairs $(u,v)$. This observation allows to conclude that performing alternating minimization algorithm for all $2^n$ starting points whose components are $\pm 1$ finds an \textit{optimal} rank-1 approximation. Furthermore, by analyzing the behaviour of signs during alternating minimization method, we propose an improvement for an exhaustive search of optimal approximation (that is, the number of runs can be significantly reduced in comparison to $2^n$). More precisely, it suffices to consider only starting points whose signs of components do not change under the iterations of the alternating minimization method. We encode the information about behaviour of signs in the form of specific directed graphs associated with the matrix $A$ and reveal the structure of these graphs.
%propose an algorithm that finds an \textit{optimal} rank-1 approximation via finite number of alternating minimization algorithm runs for special starting points. By analyzing the signs of vectors $u^{(k)}$ and $v^{(k)}$ that are generated by alternating minimization algorithm, we propose an improvement for an exhaustive search of optimal approximation. 

%We reveal several properties of alternating minimization algorithm for this problem and demonstrate how it can be used to obtain an algorithm which is capable to find optimal rank-1 approximations. As a result we propose a method which can build an optimal rank-1 Chebyshev approximation for small matrices in a feasible time.

The rest of the paper is organized as follows.
In \sect{related} we review the results concerning Chebyshev approximations known in the literature. In \sect{prelim}, we present several facts concerning definition, well-posedness, and basic properties of the alternating minimization method. In \sect{mainres}, we analyze the behaviour of signs of the components of the vectors generated by the alternating minimization method. In \sect{conv} we prove the main results concerning the convergence. %Namely, we show that the result of the alternating minimization method depends only on the signs of the components of the starting point, and study how the signs of the approximation vectors change when the alternating minimization method applied.
Finally, in \sect{numres}, we describe a numerical method that makes it possible to find optimal Chebyshev approximations of rank 1 and present the results of numerical experiments.

\section{Related work}
\label{sec:related}
As far as we know, the problem of constructing and analyzing low-rank approximations of matrices
in the Chebyshev norm has been little studied.

One of the important properties of the Chebyshev norm is that, in a sense, in that norm any matrix can be approximated by a matrix of low rank. More precisely, the following theorem is proved in \cite{udell2019big}.
\begin{theorem}
\label{theorem:townsend_theorem}
Let $X \in \mathbb{R}^{m\times n}$, where $m \ge n$ and $0 < \varepsilon < 1$. Then, with
\begin{equation*}
r = \lceil 72\log{(2n + 1)}/\varepsilon^2 \rceil
\end{equation*}
we have
\begin{equation*}
\inf\limits_{\rank Y \le r} \|X - Y\|_C \le \varepsilon \|X\|_2.
\end{equation*}
\end{theorem}
The theorem states that with a fixed accuracy of the approximation $\varepsilon$ and a bounded spectral norm of the matrix $\|X\|_2$, the rank required to obtain the accuracy $\varepsilon$ in the Chebyshev norm increases logarithmically with the growth of the matrix size.
The difference between the low-rank approximations in the Frobenius and Chebyshev norms can be seen on the example of the identity matrix. An identity matrix of size $n \times n$ can be approximated with rank $n-1$ with accuracy not less than $1$ in the Frobenius norm. At the same time, with the fixed accuracy $\varepsilon$, the rank required to obtain the Chebyshev $\varepsilon$-approximation grows logarithmically with the size of the matrix. So, for example, it is guaranteed that identity matrix of size $10^{100}$ can be approximated with an accuracy $\varepsilon = 0.01$ by a matrix of rank $r \approx 1.6\times 10^8$.

In \cite{udell2019big} the authors propose a constructive method for proving \thm{townsend_theorem}, however, it requires the construction of singular value decomposition of the matrix, which prevents the method from being applied to large matrices. At the same time, the large constant in \thm{townsend_theorem} does not give reasonable estimates for small matrices. In \cite{zamarashkin2022best}, an algorithm is proposed that is capable of constructing more efficient approximations. The authors propose a generalization of the Remez algorithm, which allows to find the optimal solution to the problem
\begin{equation*}
     \|A - U V^T\|_C \to \min\limits_{U \in \mathbb{R}^{m \times r}}, ~~~ A \in \mathbb{R}^{m \times n}, V \in \mathbb{R}^{n \times r}.
\end{equation*}
Then the authors applied alternating minimization method, that is, alternately computed $\argmin\limits_{U \in \mathbb{R}^{m \times r}} \|A - U V^T\|_C$ and $\argmin\limits_{V \in \mathbb{R}^{n \times r}} \|A - U V^T\|_C$.
% The result is an algorithm that allows in practice to obtain an accuracy even higher than that guaranteed by Theorem~\ref{townsend_theorem}. For example, when approximating with $r = \sqrt{n}$ Theorem~\ref{townsend_theorem} guarantees the estimate
% \begin{equation*}
%     \varepsilon \le \dfrac{6 \sqrt{2} \log^{0.5}{(2n + 1)}}{n^{0.25}}
% \end{equation*}
% while the algorithm proposed in \cite{TODO} in practice gives an accuracy
% \begin{equation*}
%     \varepsilon \approx \dfrac{\log^{0.6}{n}}{n^{0.5}}.
% \end{equation*}

In \cite{daugavet1971uniform}, it is also proposed to use the alternating minimization method for solving the problem of Chebyshev approximations, but only for approximations of rank 1. At the same time, the structure of the alternating minimization method in application to this problem is deeply analyzed. Among the main results proved in \cite{daugavet1971uniform}, the following can be highlighted. A necessary and sufficient condition is proved for a pair of vectors $(u, v)$ to be a local minimum of the problem
\begin{equation}
    \label{eq:main_problem}
     \|A - uv^T\|_C \to \min\limits_{u \in \mathbb{R}^m, v \in \mathbb{R}^n}.
\end{equation}
In addition, it is always possible to select a convergent subsequence from the sequence of vectors generated by the alternating minimization method, and any convergent subsequence converges to a local minimum of the problem~\eq{main_problem}. In this paper, we essentially rely on the theoretical results obtained in \cite{daugavet1971uniform}, but unlike \cite{daugavet1971uniform}, we pose ourselves the problem of finding a \textit{global} minimum of~\eq{main_problem} and develop the corresponding theory. 

The alternating minimization method was also proposed in \cite[Section 4]{gillis2019low} as an heuristic algorithm to solve~\eq{main_problem}. This work mainly concerns a different problem, namely the decision variant of~\eq{main_problem}. That is, given $k \ge 0$ determine, whether there exist $u \in \mathbb R^m$ and $v \in \mathbb R^n$ such that $\|A - u v^T\|_C \le k$. The authors prove that this problem can be solved in polynomial time if the signs of $u$ and $v$ are known, and that the general decision problem is NP-complete. However, the theoretical analysis of the structure of the alternating minimization method is not presented there. It is noteworthy that both in our paper and in \cite{gillis2019low} some results are achieved by considering some special graphs, even though the graphs are completely different. Finally, we note that the problem posed in \cite[Remark 2]{gillis2019low} concerning the signs of the components of $u$ and $v$ that give an optimal rank-1 approximation for a matrix that is sufficiently close to a rank-1 matrix is solved here in remarks in Sections~4 and~5.

\section{Preliminaries}
\label{sec:prelim}
%We are going to solve the problem \eq{main_problem} with the \textit{alternating minimization method}. That is, we begin from a vector $v^{(0)} \in \mathbb{R}^n$ and alternately solve the minimization problems with respect to $u$ and $v$ separately. Thus, we construct a pair of sequences
%\begin{equation}
%\label{eq:one_matrix_problem}
%    u^{(i)} \in \argmin\limits_{u \in \mathbb{R}^m} \|A - u (v^{(i)})^T\|_C,
%\end{equation}
%\begin{equation*}
%    v^{(i+1)} \in \argmin\limits_{v \in \mathbb{R}^n} \|A - u^{(i)} v^T\|_C.
%\end{equation*}
In this section we provide the definition of the alternating minimization method and the setting, in which it is well-defined.
Before doing so we briefly discuss the simpler approximation problem of the form
\begin{equation}
\label{eq:one_vector_problem}
     \|a - uv\|_\infty \to \min\limits_{u \in \mathbb{R}},
\end{equation}
where $a$ and $v$ are given vectors. Throughout this paper we shall use the symbol $\sign$ to denote the function defined on $\mathbb R$ by the rule 
$$
\sign(t) = \begin{cases}
1, & \text{if}\; t > 0 \\
0, & \text{if}\; t = 0 \\
-1, & \text{if}\; t < 0
\end{cases}
$$
Also, in what follows everywhere we assume that the sizes $m$ and $n$ are greater than $1$. Note that some of our results (for example, \lem{one_vector_approx} below and all facts related to the notion of alternance) are not applicable in the case, where one of the matrix' sizes is equal to $1$.
\begin{lemma}

\label{lemma:one_vector_approx}
Let $a,v \in \mathbb R^n$. Assume that all components of $v$ do not vanish. Then there exists a unique $t \in \mathbb R$ such that
\begin{equation*}
    \|a-tv\|_{\infty} = \inf\limits_{u \in \mathbb R}\|a - uv\|.
\end{equation*}
Moreover, $u = t$ if and only if there is a pair of distinct indices $i,j \in \{1, \dots, n\}$ such that
\begin{eqnarray*}
    |a_i - uv_i| = |a_j - uv_j| & = & \|a - uv\|_{\infty},\\
    \sign(v_i(a_i - uv_i)) & = & -\sign(v_j(a_j - uv_j)).
\end{eqnarray*}
\end{lemma}

Since this lemma is quite elementary we only sketch the proof. Existence of $t$ is a well-known fact about approximations on finite dimensional spaces. Uniqueness easily follows from the fact that components of $v$ do not vanish: the function $\|a - tv\|_{\infty}$ is convex and piece-wise linear with non-zero slopes. Finally, the last statement in this lemma is an immediate corollary of a special case of Lemma~3 from \cite{zamarashkin2022best} for size $1$ matrices. In \cite{zamarashkin2022best} the authors also propose an algorithm that can build the solution of the problem \eq{one_vector_problem} in a polynomial number of operations if all components of $v$ are non-zero.

\begin{definition}
Let $v \in \mathbb{R}^n$. A vector $v$ is called Chebyshev if all of its components are non-zero. We call $v$ alternance-free if there is only one number $i \in \{1, \dots, n\}$ such that $|v_i| = \|v\|_{\infty}$. We denote the set of all Chebyshev vectors in $\mathbb R^n$ by $\mathrm{Ch}_n$.
\end{definition}

Let $a,v \in \mathbb R^n$ and assume that $v$ is Chebyshev. Let $\mu(a,v)$ denote the unique real number satisfying 
\begin{equation*}
    \|a - \mu(a,v)v\|_\infty = \inf\limits_{u \in \mathbb R} \|a - uv\|_\infty.
\end{equation*}
The explicit formula for $\mu(a,v)$ given in \cite{zamarashkin2022best} implies that the mapping $\mu(a,v)$ is continuous (in both arguments). Also, if $a$ is an alternance-free vector, then we introduce $\chi(a) \in \{1, \dots, n\}$ and $\delta(a) \in \mathbb R$ such that
\begin{equation*}
    |a_{\chi(a)}| = \|a\|_\infty \quad \text{and} \quad \delta(a) = \|a\|_\infty - \max\limits_{j \ne \chi(a)}|a_j|.
\end{equation*}

\begin{theorem}
\label{theorem:preserve_chebyshev}
Let $a \in \mathbb{R}^{n}$. Then the following statements hold.
\begin{enumerate}[label=(\roman*)]
\item\label{pr_ch_i} $\mu(a,v) \ne 0$ for all Chebyshev $v \in \mathbb R^n$ if and only if $a$ is alternance-free.
\item\label{pr_ch_ii} Assume that $a$ is alternance-free and let $v \in \mathbb R^n$ be a Chebyshev vector. Then $\sign(\mu(a,v)) = \sign(a_{\chi(a)} v_{\chi(a)})$ and the inequality
\begin{equation}
\label{eq:pres_ch}
    \dfrac{\delta(a)}{2 \|v\|_\infty} \le |\mu(a,v)| \le \dfrac{2 \|a\|_\infty}{\|v\|_\infty}
\end{equation}
holds.
\end{enumerate}
\begin{proof}
Assume that $a$ is alternance-free. Then \lem{one_vector_approx} implies that $a - \mu(a,v)v$ is not alternance-free for all Chebyshev vectors $v$ and, hence, $\mu(a,v) \ne 0$. Now assume that $a$ is not alternance-free and let $i,j \in \{1, \dots, n\}$ be a distinct pair of indices such that
\begin{equation*}
    |a_i| = |a_j| = \|a\|_\infty.
\end{equation*}
Let $v \in \mathbb R^n$ denote any Chebyshev vector such that $\sign(v_i a_i) = - \sign(v_ja_j)$. Then, by \lem{one_vector_approx}, we obtain that $\mu(a,v) = 0$. Thus, the statement \ref{pr_ch_i} is verified.

Now we prove \ref{pr_ch_ii}. As $a$ is alternance-free, it follows that $\mu(a,v) \ne 0$ and, therefore,
\begin{equation*}
    \|a - \mu(a,v) v\|_\infty < \|a\|_\infty.
\end{equation*}
Thus,
\begin{equation*}
    |a_{\chi(a)} - \mu(a,v) v_{\chi(a)}| < |a_{\chi(a)}|
\end{equation*}
and this is possible only if $\sign(\mu(a,v)) = \sign(a_{\chi(a)} v_{\chi(a)})$. It remains to prove \eqref{eq:pres_ch}. Since
\begin{equation*}
    \|a - \mu(a,v) v\|_\infty \le \|a\|_\infty,
\end{equation*}
we obtain that
\begin{equation*}
    |\mu(a,v)| \|v\|_\infty - \|a\|_\infty \le \|a\|_\infty.
\end{equation*}
The second inequality in \eqref{eq:pres_ch} immediately follows. In order to prove the remaining inequality, note that \lem{one_vector_approx} implies that
\begin{equation*}
    |a_i - \mu(a,v) v_i| \ge |a_{\chi(a)} - \mu(a,v) v_{\chi(a)}|
\end{equation*}
for some $i \ne \chi(a)$. Thus,
\begin{equation*}
    |a_i| + |\mu(a,v)| |v_i| \ge |a_{\chi(a)}| - |\mu(a,v)| |v_{\chi(a)}|
\end{equation*}
and, therefore,
\begin{equation*}
    |\mu(a,v)|(|v_i| + |v_{\chi(a)}|) \ge |a_{\chi(a)}| - |a_i| \ge \delta(a).
\end{equation*}
\end{proof}
\end{theorem}

Now we are ready to introduce the alternating minimization method. Here and further we denote columns of matrices by subscripts and rows by superscripts. For example, $j$-th column of matrix $A$ is denoted with $a_j$ and $i$-th row is denoted with $a^i$. Given a matrix $A \in \mathbb R^{m \times n}$ and a vector $v \in \mathrm{Ch}_n$ we denote by $\phi(A,v) \in \mathbb R^m$ the vector such that
\begin{equation*}
    \phi(A,v)_i = \mu(a^i, v), \quad i = 1, \dots, m.
\end{equation*}
Similarly, if $u \in \mathrm{Ch}_m$, then we define$\psi(A,u) \in \mathbb R^n$ by the equalities
\begin{equation*}
    \psi(A,u)_j = \mu(a_j, u), \quad j = 1, \dots, n.
\end{equation*}
It is easy to see that $\psi(A,u) = \phi(A^T, u)$. It follows from definition that
\begin{equation*}
    \inf\limits_{u \in \mathbb R^m} \|A - uv^T\|_C = \|A - \phi(A,v)v^T\|_C
\end{equation*}
for all $v \in \mathrm{Ch}_n$. Also the similar equality holds for the mapping $\psi$. Note that continuity of $\mu$ implies continuity of the mappings $\phi$ and $\psi$.

Let $A \in \mathbb R^{m \times n}$. We say that a pair of sequences $\{v^{(k)} \in \mathrm{Ch}_n\}_{k \in \mathbb N}$ and $\{u^{(k)} \in \mathrm{Ch}_m\}_{k \in \mathbb N}$ \textit{is obtained by the alternating minimization method} for a matrix $A$ and an initial point $v^{(0)} \in \mathrm{Ch}_n$, if
\begin{equation*}
    u^{(k)} = \phi(A, v^{(k-1)}) \quad \text{and} \quad v^{(k)} = \psi(A, u^{(k)}) \quad \text{for all} \quad k \in \mathbb N.
\end{equation*}

\begin{definition}
We say that the matrix $A \in \mathbb{R}^{m \times n}$ preserves Chebyshev systems if for any Chebyshev vectors $u \in \mathbb{R}^m$ and $v \in \mathbb {R}^n$ vectors $\phi(A, v)$ and $\psi(A, u)$ are also Chebyshev vectors. By $\mathrm{PC}_{m,n}$ we denote the set of all $m\times n$ matrices that preserve Chebyshev systems.
\end{definition}

It is clear from the definition and \lem{one_vector_approx}, that if $A \in \mathbb R^{m \times n}$ preserves Chebyshev systems and $v^{(0)} \in \mathrm{Ch}_n$, then there exists a (unique) pair of sequences $\{v^{(k)}\}_{k \in \mathbb N}$ and $\{u^{(k)}\}_{k \in \mathbb N}$ that is obtained by the alternating minimization method for $A$ and the initial point $v^{(0)}$. The following lemma gives a simple characterization of matrices that preserve Chebyshev systems and also states that such matrices are quite common in a certain sense.

\begin{lemma}
    A matrix $A\in \mathbb{R}^{m \times n}$ preserves Chebyshev systems if and only if all rows and columns of $A$ are alternance-free. The set $\mathrm{PC}_{m,n}$ is open and dense in $\mathbb R^{m \times n}$. The set $\mathbb R^{m \times n} \setminus \mathrm{PC}_{m,n}$ is Lebesgue-null.
\end{lemma}
\begin{proof}
    Indeed, $A$ preserves Chebyshev systems if and only if $\mu(a^i, v)$ and $\mu(a_j, u)$ are not vanishing for all $i,j$ and all Chebyshev vectors $v \in \mathbb R^n$, $u \in \mathbb R^m$. \thm{preserve_chebyshev}~\ref{pr_ch_i} implies that this is the case if and only if $a^i$ and $a_j$ are alternance-free for all $i = 1, \dots, m$ and $j = 1, \dots, n$. From this characterization it is clear that $\mathrm{PC}_{m,n}$ is open in $\mathbb R^{m \times n}$ and that the complement $\mathbb R^{m \times n} \setminus \mathrm{PC}_{m,n}$ is contained in the union of a finite number of hyperplanes in $\mathbb R^{m\times n}$. So, the complement is Lebesgue-null and has empty interior.
\end{proof}

For the following lemma we shall introduce some 
more notation. If $v \in \mathbb R^n$ is a Chebyshev vector, then we say that $\|v\|_\infty / \min\limits_{i = 1, \dots, n} |v_i|$ is the \textit{amplitude} of $v$. We shall denote it by $\mathrm{am}(v)$.  
\begin{lemma}
\label{lemma:basic_prop}
Let $A \in \mathrm{PC}_{m,n}$ and $v^{(0)} \in \mathrm{Ch}_n$. Let the pair of sequences $\{v^{(k)}\}_{k \in \mathbb N}$ and $\{u^{(k)}\}_{k \in \mathbb N}$ be constructed by the alternating minimization method for the matrix $A$ and the initial point $v^{(0)}$. Then the following statements hold.
\begin{enumerate}[label=(\roman*)]
    \item\label{basic_prop_i} $\|A - u^{(k)} (v^{(k-1)})^T\|_C \ge \|A - u^{(k)} (v^{(k)})^T\|_C \ge \|A - u^{(k+1)} (v^{(k)})^T\|_C$ for all $k \in \mathbb N$.
    \item\label{basic_prop_ii} Let $\delta_r = \min\limits_{i = 1, \dots, m} \delta(a^i)$ and $\delta_c = \min\limits_{j = 1, \dots, n} \delta(a_j)$. Then $$\|u^{(k)}\|_\infty \|v^{(k-1)}\|_\infty \le 2 \|A\|_C,\;\; \|u^{(k)}\|_\infty \|v^{(k)}\|_\infty \le 2 \|A\|_C,$$
    $$\mathrm{am}(u^{(k)}) \le 4\|A\|_C / \delta_r,\;\; \mathrm{am}(v^{(k)}) \le 4\|A\|_C / \delta_c$$ for all $k \in \mathbb N$.
    \item\label{basic_prop_iii} If the pair of sequences $\{\tilde{v}^{(k)}\}_{k \in \mathbb N}$ and $\{\tilde{u}^{(k)}\}_{k \in \mathbb N}$ is obtained by the alternating minimization method for matrix $A$ and the initial point $cv^{(0)}$, where $c \ne 0$, then $\tilde{v}^{(k)} = cv^{(k)}$ and $\tilde{u}^{(k)} = 1/c\; u^{(k)}$.
\end{enumerate}
\begin{proof}
The equality
\begin{equation*}
    \inf\limits_{u \in \mathbb R^m} \|A - u(v^{(k)})^T\|_C = \|A - \phi(A,v^{(k)})(v^{(k)})^T\|_C
\end{equation*}
implies
\begin{equation*}
    \|A - u^{(k)} (v^{(k)})^T\|_C \ge \|A - u^{(k+1)} (v^{(k)})^T\|_C
\end{equation*}
because $u^{(k+1)} = \phi(A,v^{(k)})$. The other inequality in the statement~\ref{basic_prop_i} can be proved similarly. The inequalities of~\ref{basic_prop_ii} are the immediate consequences of \thm{preserve_chebyshev}~\ref{pr_ch_ii} and of just proven~\ref{basic_prop_i}. Finally,~\ref{basic_prop_iii} follows from an elementary observation $\mu(a,cv) = 1/c\;\mu(a,v)$.
\end{proof}
\end{lemma}
Let $A \in \mathbb R^{m \times n}$ preserve Chebyshev systems. Let $v \in \mathbb R^n$ be a Chebyshev vector and construct a pair of sequences $\{v^{(k)}\}_{k \in \mathbb N}$ and $\{u^{(k)}\}_{k \in \mathbb N}$ by the alternating minimization method for the matrix $A$ with the initial point $v^{(0)} = v$. \lem{basic_prop}~\ref{basic_prop_i} implies that the sequence $\|A - u^{(k)} (v^{(k)})^T\|_C$ decreases and, since it consists only of non-negative numbers, it converges. We shall denote its limit by $E(A,v)$. The concluding lemma of this section contains elementary properties of this function.

\begin{lemma}
\label{lemma:basic_prop_E}
Let $A \in \mathrm{PC}_{m,n}$. Then the following statements hold.
\begin{enumerate}[label=(\roman*)]
    \item\label{basic_prop_E_i} If $v \in \mathrm{Ch}_n$, then $E(A,v) \ge 0$ and $E(A,v) = E(A,cv)$ for all $c \ne 0$.
    \item\label{basic_prop_E_ii} If $v \in \mathrm{Ch}_n$, then $E(A,v) = E(A,w)$, where $w = \psi(A, \phi(A,v))$.
    \item\label{basic_prop_E_iii} The function $E(A,v)$ is upper semi-continuous with respect to $v \in \mathrm{Ch}_n$.
\end{enumerate}
\begin{proof}
    The statements~\ref{basic_prop_E_i} and~\ref{basic_prop_E_ii} immediately follow from the definition of $E(A,v)$. The upper semi-continuity holds, for $E(A,v)$ is a limit of a decreasing sequence of continuous functions.
\end{proof}
\end{lemma}

\section{Analysis of the signs in the alternating minimization method}
\label{sec:mainres}
% \begin{note}
% \label{note_on_points_in_one_octant}
% Из теоремы 5 сразу следует, что если мы имеем две стартовые точки $v_0^{(1)}$ и $v_0^{(2)}$, такие что запущенный из них метод переменных направлений сходится к парам $(\widehat{u}^{(1)}, \widehat{v}^{(1)})$ и $(\widehat{u}^{(2)}, \widehat{v}^{(2)})$ соответственно, причем знаки компонент векторов $\widehat{v}^{(1)}$ и $\widehat{v}^{(2)}$ совпадают, то
% \begin{equation*}
%     \|A - \widehat{u}^{(1)} (\widehat{v}^{(1)})^T\|_C = \|A - \widehat{u}^{(2)} (\widehat{v}^{(2)})^T\|_C
% \end{equation*}
% То есть все предельные точки метода переменных направлений, у которых совпадают знаки координат $u$ или $v$, дают одно и то же значение функционала.
% \end{note}

%It immediately follows from \thm{daugavet_main} that if we have two starting points $v_0^{(1)}$ and $v_0^{(2)}$ such that the alternating minimization initialized from them converges to the pairs $(\widehat{u }^{(1)}, \widehat{v}^{(1)})$ and $(\widehat{u}^{(2)}, \widehat{v}^{(2)})$ respectively, and the signs of the components of $\widehat{v}^{(1)}$ and $\widehat{v}^{(2)}$ coincide, then
%\begin{equation*}
 %   \|A - \widehat{u}^{(1)} (\widehat{v}^{(1)})^T\|_C = \|A - \widehat{u}^{(2)} (\widehat{v}^{(2)})^T\|_C.
%\end{equation*}
%Our next goal is to prove that this is also true if the signs of $v_0^{(1)}$ and $v_0^{(2)}$ coincide.
In this section we analyze the behaviour of the signs of components of the vectors $u^{(k)}$ and $v^{(k)}$ that were obtained by the alternating minimization method. More precisely, we prove that the signs are completely determined by the matrix and the signs of the initial point and, moreover, that the signs stabilize for large $k$ (in fact, for $k$ larger than $\min(m,n)$). 

Let $v \in \mathbb R^n$ be a Chebyshev vector. Let $\mathcal S(v)$ denote the vector with components $\mathcal S(v)_i = \sign(v_i)$, $i = 1, \dots, n$. That is, $\mathcal S$ is a mapping from $\mathrm{Ch}_n$ to $\{-1,1\}^n$. 
\begin{theorem}
\label{theorem:SignIndependence}
Let $A \in \mathrm{PC}_{m,n}$. If $v_1, v_2 \in \mathrm{Ch}_n$ and $\mathcal{S}(v_1) = \mathcal{S}(v_2)$, then
\begin{equation*}
    \mathcal{S}(\phi(A, v_1)) = \mathcal{S}(\phi(A, v_2))
\end{equation*}

Similarly, if $u_1, u_2 \in \mathrm{Ch}_m$ and $\mathcal{S}(u_1) = \mathcal{S}(u_2)$, then
\begin{equation*}
    \mathcal{S}(\psi(A, u_1)) = \mathcal{S}(\psi(A, u_2))
\end{equation*}
\begin{proof}
Let $O = \{v \in \mathrm{Ch}_n: \mathcal S(v) = \mathcal S(v_1)\}$. It is clear that $O$ is convex and, therefore, connected. The function $s(v) = \mathcal S(\phi(A,v))$ is continuous on $\mathrm{Ch}_n$, for $\phi(A,v)$ is continuous with respect to $v$ and $\mathcal S$ is locally constant (and, hence, continuous). Since the range of $s$ is discrete, it follows that $s$ is constant on $O$. Thus, $s(v_1) = s(v_2)$, as $v_1, v_2 \in O$. The other statement can be proved analogously.
%Since $\mathcal{M}(u_1) = \mathcal{M}(u_2)$, the signs of all components of the vectors $u_1$ and $u_2$ coincide and the segment connecting these points does not contain points with zero components. That is, the segment connecting $u_1$ and $u_2$ consists of Chebyshev vectors. The mapping $\psi(A, u)$ is continuous on Chebyshev vectors, therefore, under the mapping $\psi$, the segment connecting $u_1$ and $u_2$ passes into a continuous curve connecting $\psi(A, u_1)$ and $\psi(A, u_2)$. Since the matrix $A$ preserves Chebyshev systems, all components of all points of this curve are not equal to 0. Therefore, the curve does not intersect the coordinate axes and the coordinates of the points $\psi(A, u_1)$ and $\psi(A, u_2)$ have the same signs, which means that
%\begin{equation*}
%    \mathcal{N}(\psi(A, u_1)) = \mathcal{N}(\psi(A, u_2))
%\end{equation*}
\end{proof}
\end{theorem}

\thm{SignIndependence} implies that the signs of the vector components at each next step of the alternating minimization method depend only on the signs of the vector at the previous step. Thus, for a pair of initial points with the same signs the alternating minimization method generates sequences with coinciding signs.

%It worth noting that in our numerical experiments the similar structures arise also for ranks larger than 1, however, for now we do not understand them fully and leave the detailed analysis for further work.

%This implies a method for constructing an optimal approximation of rank 1. It suffices to enumerate $2^{n-1}$ initial points with values $\pm 1$ (the first component can always be chosen equal to 1).

%However, in fact, we can do better than exhaustive search. To do this, we study how the signs of the vector behave when applying the method of alternating minimization.

Now we shall analyze the behaviour of the signs more closely. For the rest of this section let us fix a matrix $A \in \mathbb R^{m \times n}$ that preserves Chebyshev systems, an initial vector $v^{(0)} \in \mathrm{Ch}_n$, and the pair of sequences $\{u^{(k)}\}_{k \in \mathbb N}$ and $\{v^{(k)}\}_{k \in \mathbb N}$ that is generated by the alternating minimization method for $A$ and $v^{(0)}$. We introduce a couple of mappings $\mathcal R: \{-1, 1\}^m \rightarrow \{-1, 1\}^n$ and $\mathcal T: \{-1, 1\}^n \rightarrow \{-1,1\}^m$ defined by the formulae
\begin{equation*}
    \mathcal{R}(p) = \mathcal{S}(\psi(A, p)),\;\;\mathcal{T}(q) = \mathcal{S}(\phi(A, q)),
\end{equation*}
where $p \in \{-1, 1\}^m$ and $q \in \{-1, 1\}^n$. Finally, let $\mathcal{V} = \mathcal R \circ \mathcal T$.
The mapping $\mathcal{V}$, thus, specifies how the signs of the vector $v^{(k+1)} = \psi(A, \phi(A, v^{(k)}))$ depend on the signs of $v^{(k)}$ during the alternating minimization method.

We use the mapping $\mathcal{V} $ to construct the \textit{sign transition graph} $G_{A}$ of the matrix $A$. The vertices of this graph are all elements of $\{-1, 1\}^n$, and there is an edge from $t_1$ to $t_2$ if and only if $\mathcal{V}(t_1) = t_2$. Let us show that the graph $G_{A}$ is a set of isomorphic trees, and each tree contains exactly one vertex $t$ such that $\mathcal{V}(t) = t$.

Let $\mathfrak j(i) = \chi(a^i)$ and $\mathfrak i(j) = \chi(a_j)$ (recall, that $\chi(a)$ for an alternance-free vector $a$ denotes the position of the maximum absolute value element).
Note that, due to the assumption on the matrix $A$, we have $|a_{i, \mathfrak j(i)}| > 0$ and $|a_{\mathfrak i(j), j}| > 0$ for all $i,j$. Since $u^{(k+1)} = \phi(A, v^{(k)})$, \thm{preserve_chebyshev}~\ref{pr_ch_ii} implies that
\begin{equation*}
    \sign{u_i^{(k+1)}} = \sign{a_{i, \mathfrak j(i)}} \sign{v_{\mathfrak j(i)}^{(k)}}.
\end{equation*}
Similarly, we have
\begin{equation*}
    \sign{v_j^{(k+1)}} = \sign{a_{\mathfrak i(j), j}} \sign{u_{\mathfrak i(j)}^{(k+1)}}.
\end{equation*}
Thus, from the last two equalities we get
\begin{equation}
\label{eq:v_sign_transition}
    \sign{v_j^{(k+1)}} = \sign{a_{\mathfrak i(j), j}} \cdot \sign{a_{\mathfrak i(j), \mathfrak j(\mathfrak i(j))}} \cdot \sign{v_{\mathfrak j( \mathfrak i(j))}^{(k)}}.
\end{equation}
The equality \eq{v_sign_transition} expresses the signs of the vector $v^{(k+1)}$ in terms of the signs of the vector $v^{(k)}$, that is, it determines the mapping $\mathcal{V}$. Let us introduce the \textit{sign dependency graph} $G^{sd}_A$ of the matrix $A$. The set of vertices of $G^{sd}_A$ is $\{1, \dots, n\}$, and there is an edge from $k$ to $l$ if and only if $k = \mathfrak j(\mathfrak i(l))$.

We consider $G_{A}$ and $G^{sd}_A$ as \textit{directed} graphs. Also the word \textit{acyclic} below means that a graph does not contain directed cycles apart from loops. A vertex $t$ of a graph $G$ such that there is a loop $t \to t$ we shall call \textit{loop vertex}. Finally, by the \textit{depth} of a graph $G$ we denote the maximal possible number $p$ such that there exists a sequence of distinct vertices $t_1, \dots, t_{p}$ such that there is an edge from $t_k$ to $t_{k+1}$ for all $k = 1, \dots, p-1$.

\begin{lemma}
\label{lemma:G_sd_lemma}
    The graph $G^{sd}_A$ fulfills the following properties.
    \begin{enumerate}[label=(\roman*)]
        \item\label{G_sd_i} For each vertex there is exactly one edge that is pointing to it.
        \item\label{G_sd_ii} There is always at least one loop in the graph. If $j$ is a loop vertex, then $\sign(v^{(k)}_j) = \sign(v^{(l)}_j)$ for all $k,l \in \mathbb N$.
        \item\label{G_sd_iii} The graph $G^{sd}_A$ is acyclic.
        \item\label{G_sd_iv} All vertices of the graph $G^{sd}_A$ are reachable from loop vertices (that is, for arbitrary vertex $j$ of $G^{sd}_A$ there is a sequence $j_1, \dots, j_k$ such that $j_k = j$, $j_1$ is a loop vertex, and there is an edge from $j_{s}$ to $j_{s+1}$ for all $s = 1, \dots, k-1$).
        \item\label{G_sd_v} Consider the following process for the matrix $|A|$ (the absolute value is taken element-wise). Take a column, find the maximum element in it, find the maximum element in the corresponding row, then again the maximum element in the corresponding column, and so on. The maximal possible number of distinct columns in a foregoing process is equal to the depth of the graph $G^{sd}_A$.
    \end{enumerate}
\begin{proof} 
    Statement \ref{G_sd_i} trivially follows from the definition of $G^{sd}_A$.
    To prove \ref{G_sd_ii} note, that a vertex $j$ is a loop vertex if and only if $\mathfrak j(\mathfrak i(j)) = j$. Thus, loop vertices are exactly the indices of columns of the matrix $A$ for which the maximum absolute value element is also the maximum absolute value element in its row. It is easy to see such an element in the matrix always exists (e.g., an element $a_{ij}$ such that $|a_{ij}| = \|A\|_C$), so the graph $G^{sd}_A$ contains at least one loop. In addition, if $j$ is a loop vertex, then the equality \eq{v_sign_transition} becomes
    \begin{equation*}
        \sign{v_j^{(k+1)}} = \sign{v_j^{(k)}}.    
    \end{equation*}

    Now we prove \ref{G_sd_iii}. It is clear that $\|a_j\|_\infty \le \|a^{\mathfrak i(j)}\|_\infty$ and ${\|a^i\|_\infty \le \|a_{\mathfrak j(i)}\|_\infty}$, whence $\|a_j\|_\infty \le \|a_{\mathfrak j(\mathfrak i(j))}\|_\infty$.
    Moreover, due to the assumptions on the matrix $A$, if $\mathfrak j(\mathfrak i(j)) \neq j$, then
    \begin{equation*}
        \|a_j\|_\infty < \|a_{\mathfrak j(\mathfrak i(j))}\|_\infty.
    \end{equation*}
    Thus, a directed cycle in $G^{sd}_A$ has to be a loop.
    
    Finally, \ref{G_sd_iv} easily follows from \ref{G_sd_i} and \ref{G_sd_iii}, and \ref{G_sd_v} is clear in view of the proof of the statement \ref{G_sd_iii}.
\end{proof}
\end{lemma}

\begin{lemma}
\label{lemma:G_V_basic}
    The graph $G_{A}$ satisfies the following properties.
    \begin{enumerate}[label=(\roman*)]
        \item\label{basicI} For each vertex in $G_{A}$ there is exactly one edge that is pointing out of it.
        \item\label{basicII} Each sequence of vertices $t_1, t_2, \dots$, such that there is an edge $t_k \to t_{k+1}$ for all $k$, stabilizes.
        \item\label{basicIII} $G_{A}$ is acyclic and its depth does not exceed the depth of $G^{sd}_A$.
    \end{enumerate}
    \begin{proof}
        Statement \ref{basicI} is an immediate consequence of the definition of $G_{A}$. To prove \ref{basicII} without loss of generality we can assume that $\mathcal S(v^{(0)}) = t_1$. Observe that $\mathcal S(v^{(k)}) = t_k$ for all $k \in \mathbb N$. The equation \eq{v_sign_transition} implies that $\mathcal S(v^{(k+1)})$ depends only on the components of the vector $\mathcal S(v^{(k)})$ with indices in the set $\{\mathfrak j(\mathfrak i(j)): j = 1, \dots, n\}$. Thus, it easily follows that $t_{p} = t_{p+1} = \dots$, where $p$ denotes the depth of $G^{sd}_A$. Indeed, vector $\mathcal S(v^{(k+p-1)})$ depends only on the components of $\mathcal S(v^{(k)})$ on the set
        \begin{equation*}
            F = \{(\mathfrak j \circ \mathfrak i)^{p-1}(j): j = 1, \dots, n\},
        \end{equation*}
        which contains exactly loop vertices of $G^{sd}_A$ by \lem{G_sd_lemma}~\ref{G_sd_v}. Observe that \lem{G_sd_lemma}~\ref{G_sd_ii} implies that the components of vectors $\mathcal S(v^{(k)})$ on $F$ do not depend on $k$, implying that $t_{k} = t_{p}$ for all $k \ge p$. Hence, the proof of \ref{basicII} is complete, and it is also evident that the depth of $G_{A}$ does not exceed $p$. The acyclicity of $G_{A}$ easily follows from \ref{basicI} and \ref{basicII}.
    \end{proof}
\end{lemma}

\lem{G_V_basic} implies that for each vertex $t \in \{-1, 1\}^n$ there is a unique sequence $t_1, t_2, \dots$ such that $t = t_1$ and for all $k \in \mathbb N$ there is an edge $t_k \to t_{k+1}$. Since this sequence stabilizes, it is possible to define $f(t)$ as the vector which is equal to $t_k$ for arbitrary big $k$. By definition, for all $t$ there is a path in $G_{A}$ from $t$ to $f(t)$, and $f(t)$ is a loop vertex. In what follows the term {\it{connected component}} of a directed graph refers to weakly connected components (i.e. we allow to connect vertices with paths regardless of edge direction).

\begin{lemma}
The following statements hold.
\label{lemma:G_V_components}
    \begin{enumerate}[label=(\roman*)]
    \item\label{comp_i} A pair of vertices $t_1$ and $t_2$ in $G_{A}$ belong to the same connected component if and only if $f(t_1) = f(t_2)$.
    \item\label{comp_ii} If $t_1$ and $t_2$ are loop vertices in $G_{A}$ and $(t_1)_l = (t_2)_l$ for all loop vertices $l$ in $G^{sd}_A$, then $t_1 = t_2$.
    \item\label{comp_iii} Let $t_1$ and $t_2$ be loop vertices in $G_A$. Let $l$ and $j$ be vertices in $G^{sd}_A$ such that $l$ is a loop vertex and $j$ is reachable from $l$ (see \lem{G_sd_lemma}~\ref{G_sd_iii}). Then
    \begin{equation*}
        (t_1)_l / (t_1)_j = (t_2)_l / (t_2)_j.
    \end{equation*}
    \end{enumerate}
\end{lemma}
\begin{proof}
To prove \ref{comp_i} assume that $f(t_1) = f(t_2)$. Then $t_1$ and $t_2$ belong to the same connected component, since both $t_1$ and $t_2$ can be connected with $f(t_1) = f(t_2)$ with a path. The converse easily follows from the following observation: if there is an edge $t_1 \to t_2$ or $t_2 \to t_1$, then $f(t_1) = f(t_2)$. 

Now let $t_1$ and $t_2$ satisfy the assumptions of \ref{comp_ii}. Let $F_k \subset \{1, \dots, n\}$ denote the set of vertices $l$ of $G^{sd}_A$ such that there is a directed path from a loop vertex to $l$ whose length does not exceed $k$. So, $F_0$ consists of all loop vertices of $G^{sd}_A$. \lem{G_sd_lemma}~\ref{G_sd_iv} implies that
\begin{equation*}
    \bigcup\limits_{k \ge 0} F_k = \{1, \dots, n\}.
\end{equation*}
By the assumption we have $(t_1)_j = (t_2)_j$ for all $j \in F_0$. We finish the proof by showing that the equality $(t_1)_j = (t_2)_j$ for all $j \in F_k$ implies $(t_1)_j = (t_2)_j$ for all $j \in F_{k+1}$. Indeed, if $j \in F_{k+1}$, then by definition there exists $s \in F_{k}$ such that $\mathfrak j(\mathfrak i(j)) = s$. Since $t_1$ is a loop vertex, \eq{v_sign_transition} implies that $(t_1)_j = q (t_1)_s$, where $q = \sign{a_{\mathfrak i(j), j}} \cdot \sign{a_{\mathfrak i(j), s}}$. By the same reason, $(t_2)_j = q (t_2)_s$. Finally, since $s \in F_k$, we have $(t_1)_s = (t_2)_s$ and, consequently, $(t_1)_j = (t_2)_j$.

Finally, we prove \ref{comp_iii}. Let $j_1, \dots, j_k$ be chosen such that $j_1 = l$, $j_k = j$, and for all $s = 1, \dots, k-1$ there is an edge from $j_s$ to $j_{s+1}$. It is clear that
\begin{equation*}
(t_1)_l/(t_1)_{j_1} = (t_2)_l/(t_2)_{j_1}.
\end{equation*}
We prove that the equality
\begin{equation*}
(t_1)_l/(t_1)_{j_s} = (t_2)_l/(t_2)_{j_s}
\end{equation*}
implies
\begin{equation*}
(t_1)_l/(t_1)_{j_{s+1}} = (t_2)_l/(t_2)_{j_{s+1}}
\end{equation*}
(this observation proves~\ref{comp_iii} by induction). Indeed, from \eq{v_sign_transition} it follows that
\begin{equation*}
    (t_1)_{j_{s+1}} = \sign a_{\mathfrak i(j_{s+1}), j_{s+1}} \cdot \sign a_{\mathfrak i(j_{s+1}), j_s} \cdot (t_1)_{j_s}.    
\end{equation*}
It follows that
\begin{equation*}
(t_1)_{j_s} / (t_1)_{j_{s+1}} = \sign a_{\mathfrak i(j_{s+1}), j_{s+1}} \cdot \sign a_{\mathfrak i(j_{s+1}), j_s}.
\end{equation*}
Obviously, for $t_2$ the same equality
\begin{equation*}
(t_2)_{j_s} / (t_2)_{j_{s+1}} = \sign a_{\mathfrak i(j_{s+1}), j_{s+1}} \cdot \sign a_{\mathfrak i(j_{s+1}), j_s}
\end{equation*}
holds. Thus, combining the obtained equalities
\begin{equation*}
    (t_1)_l/(t_1)_{j_s} = (t_2)_l/(t_2)_{j_s} \quad \text{and} \quad (t_1)_{j_s} / (t_1)_{j_{s+1}} = (t_2)_{j_s} / (t_2)_{j_{s+1}}
\end{equation*}
we get that $(t_1)_l / (t_1)_{j_{s+1}} = (t_2)_l / (t_2)_{j_{s+1}}$.
\end{proof}

Before proceeding to the concluding theorem on the structure of $G_{A}$ we define one auxiliary function $d:\{-1, 1\}^n \to \{-1, 1\}^n$ by the equality
\begin{equation*}
    d(t)_i = t_i/f(t)_i.
\end{equation*}
That is, $d(t)_i$ is equal to $1$ if $t_i$ coincides with $f(t)_i$ and $-1$ otherwise. Note that we do not interpret the value $d(t)$ as a vertex of $G_{A}$. The equality \eq{v_sign_transition} applied simultaneously to $t$ and $f(t)$ (in view of the fact that $\mathcal V(f(t)) = f(t)$) implies that 
\begin{equation}
\label{eq:basic_property_d}
    d(\mathcal V(t))_j = d(t)_{\mathfrak j (\mathfrak i(j))}
\end{equation}
for all $j \in \{1, \dots, n\}$.
Also, the definition of $d$ implies that $d(t)_j = 1$ for all loop vertices $j$ of $G_A^{sd}$. Finally, $t \in \{-1, 1\}^n$ is a loop vertex of $G_A$ if and only if $t = f(t)$ which is equivalent to the fact that $d(t)_i = 1$ for all $i$.

\begin{theorem}\label{theorem:sign_thm}
Let the matrix $A \in \mathbb{R}^{m \times n}$ preserve Chebyshev systems. Denote by $F$ the set of all numbers $j \in \{1, \dots, n\}$ such that the maximum absolute value element in the column $a_j$ is also the maximum absolute value element in its row. Let $k = |F|$. Then the following statements hold.
\begin{enumerate}[label=(\roman*)]
    \item\label{thm_i} A pair of vertices $t_1, t_2 \in \{-1, 1\}^n$ of $G_{A}$ belong to the same connected component if and only if $(t_1)_j = (t_2)_j$ for all $j \in F$.
    \item\label{thm_ii} Each connected component of $G_{A}$ has $2^{n-k}$ elements (so there are exactly $2^k$ components in $G_{A}$). Also each connected component is a tree and contains exactly one loop vertex.
    \item\label{thm_iii} All components of $G_{A}$ are isomorphic. More precisely, let $C_1$ and $C_2$ be a pair of connected components. Then for arbitrary vertex $t_1$ of $C_1$ there is exactly one vertex $t_2$ of $C_2$ such that $d(t_2) = d(t_1)$. The mapping $g$, that maps vertices of $C_1$ into vertices of $C_2$ and satisfies $d(g(t)) = d(t)$ for all vertices $t$ of $C_1$, is well-defined and is an isomorphism from $C_1$ to $C_2$.
    
    %Let $f(t)$ denote the root of the tree to which the vertex $t$ belongs. Then
    %\begin{equation*}
      % % d(t)_i = \begin{cases}
       % 1, & t_i = f(t)_i \\
       % 0, & t_i \neq f(t)_i 
     %   \end{cases}
   % \end{equation*}
    \item\label{thm_iv} The depth of $G_{A}$ is equal to the depth of arbitrary connected component of $G_{A}$ and is also equal to the depth of $G^{sd}_A$ (which was computed in the statement \ref{G_sd_v} of \lem{G_sd_lemma}).
\end{enumerate}

\begin{proof}
    In view of \lem{G_V_components}, \ref{thm_i} is equivalent to the following statement: $f(t_1) = f(t_2)$ if and only if $(t_1)_j = (t_2)_j$ for all $j \in F$. Assume that ${f(t_1) = f(t_2)}$. \lem{G_sd_lemma}~\ref{G_sd_ii} implies that $(t_1)_j = f(t_1)_j$ and $(t_2)_j = f(t_2)_j$ for all $j \in F$, since $F$ is exactly the set of all loop vertices in $G^{sd}_A$. Thus, $(t_1)_j = (t_2)_j$ for all $j \in F$. Now assume that $(t_1)_j = (t_2)_j$ for all $j \in F$. By the same argument as above, we have
    \begin{equation*}
        f(t_1)_j = (t_1)_j = (t_2)_j = f(t_2)_j
    \end{equation*}
    for all $j \in F$. Thus, \lem{G_V_components}~\ref{comp_ii} implies that $f(t_1) = f(t_2)$, as $f(t_1)$ and $f(t_2)$ are loop vertices of $G_A$.

    The number of elements in a connected component is equal to the number of mappings from $\{1, \dots, n\} \setminus F$ to $\{-1, 1\}$, that is to the number $2^{n-k}$. Also, \lem{G_V_basic}~\ref{basicI} and~\ref{basicIII} imply that all components of $G_{A}$ are trees. Finally, it is easy to see that each component indeed contains at least one loop vertex (take arbitrary vertex $t$ and observe that $f(t)$ is a loop vertex that belongs to the same component). Thus, if $t_1$ and $t_2$ are loop vertices that belong to the same component, then
    \begin{equation*}
        t_1 = f(t_1) = f(t_2) = t_2.
    \end{equation*}
    Hence, \ref{thm_ii} is proved.

    Before proving \ref{thm_iii} note that the range of $d$ is contained in the set
    \begin{equation*}
        D = \{t \in \{-1, 1\}^n: t_l = 1\; \forall l \in F\},
    \end{equation*}
    which has exactly $2^{n-k}$ elements. By definition it is also clear that $d$ is injective on each connected component of $G_{A}$, which has the same number of elements. Thus, $d$ maps each component bijectively onto $D$. Hence, if $C_1$ and $C_2$ are connected components of $G_{A}$, then there exists a unique mapping $g$ that maps vertices of $C_1$ into vertices of $C_2$ and satisfies $d(g(t)) = d(t)$ for all vertices $t$ of $C_1$. It is also clear that $g$ is bijective. Thus, it remains to show that $g$ is a graph isomorphism, i.e. if there is an edge from $t_1$ to $t_2$, then there is an edge from $g(t_1)$ to $g(t_2)$ (note that the converse is not necessary to prove, since $C_1$ and $C_2$ are interchangeable). Observe, that \eq{basic_property_d} implies that $d(\mathcal{V}(t_1)) = d(\mathcal{V}(t_2))$ when $d(t_1) = d(t_2)$. Finally, assume that there is an edge from $t_1$ to $t_2$, where $t_1$ and $t_2$ are vertices in $C_1$. Then $\mathcal V(t_1) = t_2$ and, therefore,
    \begin{equation*}
        d(\mathcal V(g(t_1))) = d(\mathcal V(t_1)) = d(t_2) = d(g(t_2)).
    \end{equation*}
    Thus, $\mathcal V(g(t_1)) = g(t_2)$, so $g(t_1)$ and $g(t_2)$ are connected with an edge.

    Since all components of $G_{A}$ have the same depth, it remains to prove that the depth of $G_{A}$ is equal to the depth of $G^{sd}_A$. Let $p$ denote the depth of $G^{sd}_A$. In the statement \ref{basicIII} of \lem{G_V_basic} it is already proved that the depth of $G_{A}$ does not exceed $p$. Let $j_1, \dots, j_p$ be distinct elements of $\{1, \dots, n\}$ such that $\mathfrak j(\mathfrak i(j_k)) = j_{k-1}$ for all $k = 2, \dots, p$. That is, $j_1, \dots, j_p$ is one of the longest possible paths in $G^{sd}_A$ (in this case $j_1$ is a loop vertex). Consider arbitrary $t \in \{-1, 1\}^n$ such that $d(t)_{j_2} = -1$. Applying the equality \eq{basic_property_d} $k$ times we get that $d(\mathcal V^k(t))_{j_{k+2}} = -1$ for $k = 0, 1, \dots, p-2$. Therefore, $\mathcal V^k(t)$ is not a loop vertex for $k = 0,1, \dots, p-2$ and acyclicity of $G_{A}$ implies that $t, \mathcal V(t), \dots, \mathcal V^{p-1}(t)$ are distinct. Thus, the depth of $G_{A}$ is greater or equal then $p$.
\end{proof}
\end{theorem}

\begin{figure}
\centering
\begin{subfigure}{.5\textwidth}
  \centering
  \includegraphics[width=.98\linewidth]{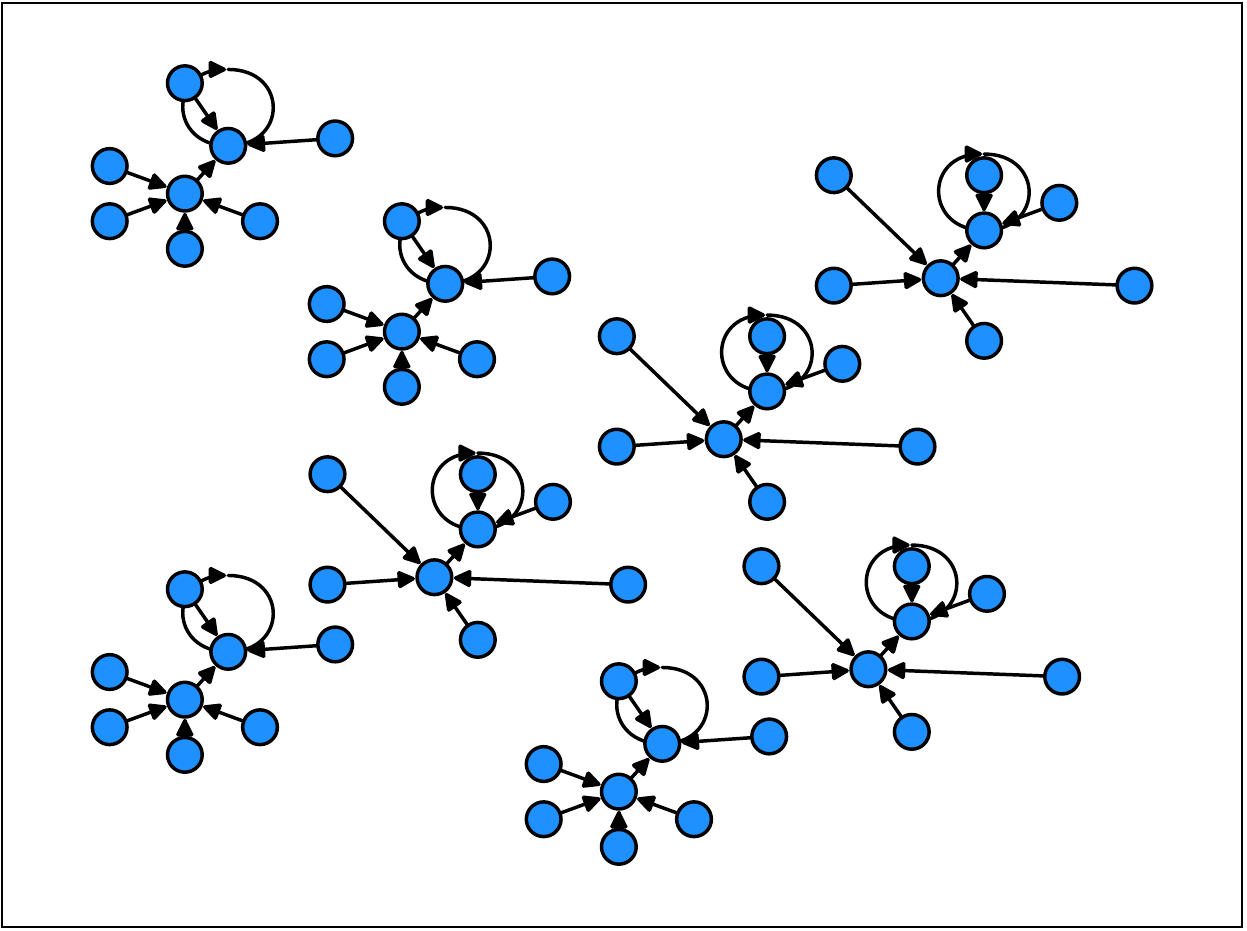}
%   \caption{A subfigure}
%   \label{fig:sub1}
\end{subfigure}%
\begin{subfigure}{.5\textwidth}
  \centering
  \includegraphics[width=.98\linewidth]{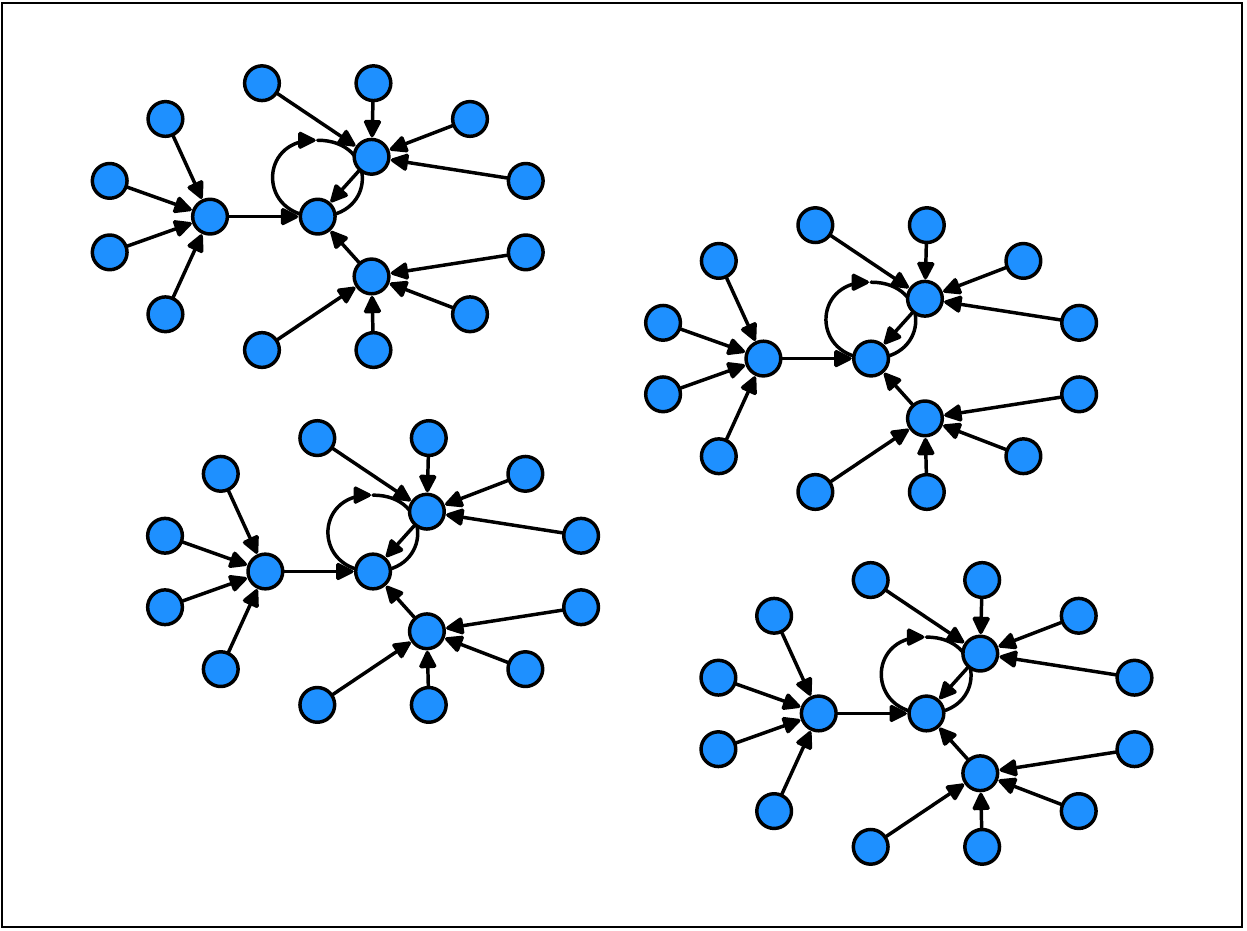}
%   \caption{A subfigure}
%   \label{fig:sub2}
\end{subfigure}
\caption{Examples of sign transition graphs for random matrices.}
\label{fig:graph_examples}
\end{figure}

Examples of sign transition graphs are shown in \fig{graph_examples}.

\begin{remark}
    1. It is easy to describe connected components of the set $\mathrm{PC}_{m,n}$. Namely, each of the components is convex and matrices ${A,B \in \mathrm{PC}_{m,n}}$ belong to the same component if and only if $\chi(a^i) = \chi(b^i)$ and
    \begin{equation*}
        \sign(a_{i,\chi(a^i)}) = \sign(b_{i,\chi(b^i)})
    \end{equation*}
    for all $i = 1, \dots, m$ and $\chi(a_j) = \chi(b_j)$ and
    \begin{equation*}
        \sign(a_{\chi(a_j), j}) = \sign(b_{\chi(b_j), j})
    \end{equation*}
    for all $j = 1, \dots, n$. It is clear that for matrices $A,B \in \mathrm{PC}_{m,n}$ that belong to the same component graphs $G_A$ and $G_B$ coincide, so the behaviour of signs during the alternating minimization method is the same for $A$ and $B$. 
    
    2. Connected components of $\mathrm{PC}_{m,n}$ that contain rank-1 matrices are especially easy to describe. Indeed, if $A \in \mathrm{PC}_{m,n}$, then the connected component of $A$ contains a rank-1 matrix if and only if $\chi(a^{i_1}) = \chi(a^{i_2})$ and $\chi(a_{j_1}) = \chi(a_{j_2})$ for all $i_1,i_2 = 1, \dots, m$ and $j_1, j_2 = 1, \dots, n$. That is, for such matrix $G^{sd}_A$ and $G_A$ have depth $2$, and $G_A$ has two loop vertices (or, equivalently, two connected components). Moreover, during the alternating minimization method after the first iteration the signs of vectors stabilize. 
\end{remark}

\section{Convergence of the alternating minimization method}
\label{sec:conv}
In this section we prove the following result.
\begin{theorem}
    \label{theorem:mainth}
    Let $A \in \mathbb R^{m \times n}$ preserve Chebyshev systems and $v,\tilde{v} \in \mathrm{Ch}_n$. Assume that $\mathcal S(v) = \mathcal S(\tilde{v})$. Then $E(A,v) = E(A,\tilde{v})$.
\end{theorem}

Now we introduce some notation. Given $A \in \mathbb R^{m \times n}$, $u \in \mathbb R^m$, and $v \in \mathbb R^n$ let
$$ T(A,u,v) = \{(i,j): |a_{ij} - u_iv_j| = \|A - uv^T\|_C\},
$$
$$ R(A,u,v) = \{i: (i,j) \in T(A,u,v)\; \text{for some}\; j\},
$$
$$ C(A,u,v) = \{j: (i,j) \in T(A,u,v)\; \text{for some}\; i\}.
$$
$$ \tau_{ij}(A,u,v) = \sign(u_i v_j (a_{ij} - u_iv_j)),
$$

%we denote the set $\{(i,j): |a_{ij} - u_iv_j| = \|A - uv^T\|_C\}$ by $T(A,u,v)$. Also let $\tau_{ij}(A,u,v) = \sign(u_i v_j (a_{ij} - u_iv_j))$. By $R(A,u,v)$ and $C(A,u,v)$ we denote the sets $\{i: (i,j) \in T(A,u,v)\; \text{for some}\; j\}$ and $\{j: (i,j) \in T(A,u,v)\; \text{for some}\; i\}$.

\begin{definition}
    Let $A \in \mathbb R^{m \times n}$, $u \in \mathrm{Ch}_m$, and $v \in \mathrm{Ch}_n$. We say that a sequence of $2k$ $(k \ge 2)$ distinct points
\begin{equation*}
    (i_1, j_1), (i_1, j_2), (i_2, j_2), \dots, (i_k, j_k), (i_k, j_1) \in T(A,u,v)
\end{equation*}
is a two-dimensional alternance for a triple $(A,u,v)$, if
\begin{multline*}
    \tau_{i_1 j_1}(A,u,v) = -\tau_{i_1 j_2}(A,u,v) = \tau_{i_2 j_2}(A,u,v) = \dots =\\
    \tau_{i_k j_k}(A,u,v) = -\tau_{i_k j_1}(A,u,v),
\end{multline*}
\end{definition}

Our method of proving \thm{mainth} will be based on the following fact that can be found in \cite[Section~1]{daugavet1971uniform}.

\begin{theorem}
\label{theorem:daugavet_main}
(Daugavet). Let $A \in \mathbb{R}^{m \times n}$, $v \in \mathrm{Ch}_n$, and $u \in \mathrm{Ch}_m$. Assume that the sequence $(i_1, j_1), (i_1, j_2), (i_2, j_2), \dots, (i_k, j_k), (i_k, j_1)$ is a two-dimensional alternance for a triple $(A,u,v)$.
% Let in the sequence $(u^{(k)}, v^{(k)})$ all vectors are Chebyshev. Let also %$u^{(\alpha_k)}$ and $v^{(\alpha_k)}$ be subsequences converging to Chebyshev %vectors $\widehat{u}$ and $\widehat{v}$ respectively. Then $(\widehat{u}, %\widehat{v})$ is a local minimum of the functional $\|A - u v^T\|_C$.

%Moreover, in this case the matrix $A - \widehat{u} \widehat{v}^T$ has a two-dimensional alternance. Denote the positions of the alternance
%\begin{equation*}
%    (i_1, j_1), (i_1, j_2), (i_2, j_2), \dots, (i_k, j_k), (i_k, j_1),
%\end{equation*}
Then if $\tilde{v} \in \mathbb R^{n}$ and $\tilde{u} \in \mathbb{R}^m$ and either
\begin{equation*}
    \sign{(u_i)} = \sign{(\tilde{u}_i)}, ~~~ i \in \{i_1, i_2, \dots, i_k\},
\end{equation*}
or \begin{equation*}
    \sign{(v_j)} = \sign{(\tilde{v}_j)}, ~~~ j \in \{j_1, j_2, \dots, j_k\},
\end{equation*}
%then $\|A - u v^T\|_C \ge \|A - \widehat{u} \widehat{v}^T\|_C$ for any $v \in \mathbb{R}^n$.

%Similarly, if $v \in \mathbb{R}^n$ is such that

then $\|A - \tilde{u} \tilde{v}^T\|_C \ge \|A - u v^T\|_C$.
\end{theorem}

\begin{lemma}\label{lemma:alt_basic}
    Let $A \in \mathbb R^{m \times n}$ preserve Chebyshev systems and $v \in \mathrm{Ch}_n$. Let $u = \phi(A,v)$, $\tilde{v} = \psi(A,u)$, and $\tilde{u} = \phi(A,\tilde{v})$. Then the following statements hold.
    \begin{enumerate}[label=(\roman*)]
        \item\label{alt_basic_i} For all $i \in R(A,u,v)$ there exist two distinct $j_1,j_2$ such that
        \begin{equation*}
            (i,j_1), (i,j_2) \in T(A,u,v) \quad \text{and} \quad \tau_{i j_1}(A,u,v) = -\tau_{i j_2}(A,u,v).
        \end{equation*}
        \item\label{alt_basic_ii} Assume that $\|A - uv^T\|_C = \|A - u \tilde{v}^T\|_C$. Then $T(A,u,\tilde{v}) \subset T(A,u,v)$. Moreover, $j \in C(A,u,\tilde{v})$ if and only if $j \in C(A,u,v)$ and $v_j = \tilde{v}_j$ and this is the case if and only if there exist two distinct $i_1, i_2$ such that $(i_1,j), (i_2,j) \in T(A,u,v)$ and $\tau_{i_1 j}(A,u,v) = -\tau_{i_2 j}(A,u,v)$.
        \item\label{alt_basic_iii} Assume that $\|A - uv^T\|_C = \|A - u \tilde{v}^T\|_C$ and $T(A,u,\tilde{v}) = T(A,u,v)$. Then the triple $(A,u,v)$ admits a two-dimensional alternance.
        \item\label{alt_basic_iv} Assume that $\|A - \tilde{u}\tilde{v}^T\|_C = \|A - u v^T\|_C$ and that the triple $(A,\tilde{u}, \tilde{v})$ admits a two-dimensional alternance. Then the triple $(A,u,v)$ also admits a two-dimensional alternance (more precisely, the alternance for $(A,\tilde{u}, \tilde{v})$ is an alternance for $(A,u,v)$).
    \end{enumerate}
    \begin{proof}
        The statement \ref{alt_basic_i} follows from \lem{one_vector_approx} since $u_i = \mu(a^i, v)$. To prove \ref{alt_basic_ii} consider an index $j$. Note that
        \begin{equation*}
            \|a_j - v_ju\|_\infty \ge \|a_j - \tilde{v}_j u\|_\infty
        \end{equation*}
        as $\tilde{v}_j = \mu(a_j, u)$. Thus, if $j \notin C(A,u,v)$, then
        \begin{equation*}
            \|a_j - \tilde{v}_j u\|_\infty < \|A - u v^T\|_C
        \end{equation*}
        and, therefore, $j \notin C(A,u,\tilde{v})$. Also, if $\tilde{v}_j \ne v_j$, then
        \begin{equation*}
            \|a_j - \tilde{v}_j u\|_\infty < \|a_j - v_j u\|_\infty \le \|A - u v^T\|_C.
        \end{equation*}
        Thus, if $j \in C(A,u,\tilde{v})$, then $j \in C(A,u,v)$ and $\tilde{v}_j = v_j$. But if $j \in C(A,u,v)$ and $\tilde{v}_j = v_j$, then
        \begin{equation*}
            \|a_j - \tilde{v}_j u\| = \|a_j - v_j u\| = \|A - u v^T\|_C,
        \end{equation*}
        which implies that $j \in C(A,u,\tilde{v})$. \lem{one_vector_approx} implies that $v_j = \tilde{v}_j$ and ${j \in C(A,u,v)}$ if and only if there exist two distinct $i_1, i_2$ such that
        \begin{equation*}
            (i_1,j), (i_2,j) \in T(A,u,v) \quad \text{and} \quad \tau_{i_1 j}(A,u,v) = -\tau_{i_2 j}(A,u,v).
        \end{equation*}
        It remains to prove that ${T(A,u,\tilde{v}) \subset T(A,u,v)}$. Consider $(i,j) \in T(A,u,\tilde{v})$. Then $j \in C(A,u,\tilde{v})$ and, therefore, $\tilde{v}_j = v_j$. Therefore, $a_{ij} - u_i v_j = a_{ij} - u_i \tilde{v}_j$, so $|a_{ij} - u_iv_j| = \|A - u v^T\|_C$ and ${(i,j) \in T(A,u,v)}$.

        Now assume that $T(A,u,v) = T(A,u,\tilde{v})$. Let $(i,j) \in T(A,u,v)$. By statement~\ref{alt_basic_i} there exist two distinct $j_1,j_2$ such that
        \begin{equation*}
            (i,j_1), (i,j_2) \in T(A,u,v) \quad \text{and} \quad \tau_{i j_1}(A,u,v) = -\tau_{i j_2}(A,u,v).
        \end{equation*}
        It follows that for either $\tilde{j} = j_1$, or $\tilde{j} = j_2$, we have $\tau_{ij}(A,u,v) = -\tau_{i\tilde{j}}(A,u,v)$. By the same reasoning using~\ref{alt_basic_ii} for all $(i,j) \in T(A,u,v) = T(A,u,\tilde{v})$ there exists $\tilde{i}$ such that $(\tilde{i}, j) \in T(A,u,v)$ and $\tau_{ij}(A,u,v) = -\tau_{\tilde{i}j}(A,u,v)$. The existence of an alternance for $(A,u,v)$ now follows from an obvious induction argument in view of the finiteness of $T(A,u,v)$. Thus,~\ref{alt_basic_iii} is proved.

        Finally, assume that $\|A - u v^T\|_C = \|A - \tilde{u}\tilde{v}^T\|_C$. It is clear that
        \begin{equation*}
            \|A - u v^T\|_C = \|A - u\tilde{v}^T\|_C = \|A - \tilde{u}\tilde{v}^T\|_C,
        \end{equation*}
        since by definition
        \begin{equation*}
            \|A - u v^T\|_C \ge \|A - u\tilde{v}^T\|_C \ge \|A - \tilde{u}\tilde{v}^T\|_C.
        \end{equation*}
        By applying~\ref{alt_basic_ii} twice (for $A$ and $A^T$) we obtain that $T(A,\tilde{u},\tilde{v}) \subset T(A,u,v)$, $u_i = \tilde{u}_i$ for all $i \in R(A,\tilde{u}, \tilde{v})$, and $v_j = \tilde{v}_j$ for all $j \in C(A,\tilde{u}, \tilde{v})$. Thus, $\tau_{ij}(A,\tilde{u},\tilde{v}) = \tau_{ij}(A,u,v)$ for all $(i,j) \in T(A,\tilde{u}, \tilde{v})$. Now it is obvious that a two-dimensional alternance for $(A,\tilde{u}, \tilde{v})$ is an alternance for $(A,u,v)$.
    \end{proof}
\end{lemma}
    
\begin{lemma}\label{lemma:existence_alternance}
    Let $A \in \mathrm{PC}_{m,n}$ and $v \in \mathrm{Ch}_n$. Then the following statements hold.
    \begin{enumerate}[label=(\roman*)]
        \item\label{ex_alt_i} Let $\{u^{(k)}\}_{k \in \mathbb N}$ and $\{v^{(k)}\}_{k \in \mathbb N}$ be constructed by the alternating minimization method for $A$ and the initial point $v^{(0)} = v$. Then arbitrary limit point $w$ of the sequence $w_k = v^{(k)} / \|v^{(k)}\|_\infty$ is a Chebyshev vector satisfying
        \begin{equation*}
            E(A,w) = \|A - \phi(A,w)w^T\|_C = E(A,v).
        \end{equation*}
        \item\label{ex_alt_ii} If $\|A - \phi(A,v)v^T\|_C = E(A,v)$, then the triple $(A,\phi(A,v), v)$ admits a two-dimensional alternance.
    \end{enumerate}
    \begin{proof}
        Let $w$ be the limit of a subsequence $w_{l_k}$ of $w_k$. Since the amplitude of a Chebyshev vector does not change under the multiplication by a non-zero constant, \lem{basic_prop}~\ref{basic_prop_ii} implies that $\mathrm{am}(w_k) \le C$, where $C > 0$ is some constant that depends only on the matrix $A$. It is obvious that a convergent sequence of Chebyshev vectors with bounded amplitude converges either to a zero vector, or a Chebyshev vector. Since, $\|w_k\|_\infty = 1$, it follows that $w$ is a Chebyshev vector. Finally, $E(A, w_k) = E(A,v)$ for all $k \in \mathbb N$ and the upper semi-continuity of $E$ (see \lem{basic_prop_E}~\ref{basic_prop_E_iii}) implies ${E(A,w) \ge E(A,v)}$. Moreover, $\|A - \phi(A,w)w^T\|_C \ge E(A,w)$
        and
        \begin{multline*}
            \|A - \phi(A,w)w^T\|_C = \lim \|A - \phi(A,w_{l_k}) w_{l_k}^T\|_C =\\
            \lim \|A - \phi(A,v^{(l_k)}) (v^{(l_k)})^T\|_C = E(A,v).
        \end{multline*}
        Thus,
        \begin{equation*}
            E(A,v) = \|A - \phi(A,w)w^T\|_C \ge E(A,w) \ge E(A,v)
        \end{equation*}
        and~\ref{ex_alt_i} is proved.

        Now assume that $\|A - \phi(A,v)v^T\|_C = E(A,v)$ and let $\{u^{(k)}\}_{k \in \mathbb N}$ and $\{v^{(k)}\}_{k \in \mathbb N}$ be constructed by the alternating minimization method for $A$ and the initial point $v^{(0)} = v$. Obviously,
        \begin{equation*}
            \|A - u^{(k)}(v^{(k-1)})^T\|_C = \|A - u^{(k)}(v^{(k)})^T\|_C = \|A - u^{(k+1)}(v^{(k)})^T\|_C
        \end{equation*}
        for all $k \in \mathbb N$. Thus, \lem{alt_basic}~\ref{alt_basic_ii} implies that $$T(A,u^{(1)}, v^{(0)}) \supset T(A,u^{(1)}, v^{(1)}) \supset T(A,u^{(2)}, v^{(1)}) \supset T(A,u^{(2)}, v^{(2)}) \supset \dots$$
        Since all sets in this sequence are finite and non-empty, there is $k \in \mathbb N$ such that $T(A, u^{(k+1)}, v^{(k)}) = T(A, u^{(k+1)}, v^{(k+1)})$. \lem{alt_basic}~\ref{alt_basic_iii} implies that the triple $(A, u^{(k+1)}, v^{(k)})$ admits a two-dimensional alternance. Applying \lem{alt_basic}~\ref{alt_basic_iv} $k$ times we obtain that $(A, u^{(1)}, v^{(0)}) = (A,\phi(A,v), v)$ also admits a two-dimensional alternance. 
    \end{proof}
\end{lemma}

\begin{proof}[Proof of Theorem~\ref{theorem:mainth}]
    Consider $v$ and $\tilde{v}$ from $\mathrm{Ch}_n$ such that $\mathcal S(v) = \mathcal S(\tilde{v})$. Let $\{u^{(k)}\}_{k \in \mathbb N}$ and $\{v^{(k)}\}_{k \in \mathbb N}$ (resp. $\{\tilde{u}^{(k)}\}_{k \in \mathbb N}$ and $\{\tilde{v}^{(k)}\}_{k \in \mathbb N}$) be constructed by the alternating minimization method for $A$ with the initial point ${v^{(0)} = v}$ (resp. $\tilde{v}^{(0)} = \tilde{v}$). Let $w$ and $\tilde{w}$ be some limit points of the sequences ${w_k = v^{(k)}/\|v^{(k)}\|_\infty}$ and $\tilde{w}_k = \tilde{v}^{(k)}/\|\tilde{v}^{(k)}\|_\infty$ respectively. By \lem{existence_alternance} we have that $w,\tilde{w} \in \mathrm{Ch}_n$ and
    \begin{equation*}
        E(A,v) = \|A - \phi(A,w) w^T\|_C, \quad E(A,\tilde{v}) = \|A - \phi(A,\tilde{w}) \tilde{w}^T\|_C,
    \end{equation*}
    and that triples $(A,\phi(A,w), w)$ and $(A, \phi(A, \tilde{w}), \tilde{w})$ admit two-dimensional alternances. Also note that $\mathcal S(w) = \mathcal S(\tilde{w})$. Indeed, by \thm{SignIndependence} we have $\mathcal S(w_k) = \mathcal S(\tilde{w}_k)$ for all $k \in \mathbb N$. Also, \lem{G_V_basic}~\ref{basicII} implies that $\mathcal S(w_k) = \mathcal S(w_{k+1})$ and $\mathcal S(\tilde{w}_k) = \mathcal S(\tilde{w}_{k+1})$ for large $k$. Thus, $\mathcal S(w)$ coincides with $\mathcal S(\tilde{w})$ and also coincides with $\mathcal S(w_k)$ for large $k$. Finally, \thm{daugavet_main} implies that
    \begin{equation*}
        \|A - \phi(A,w) w^T\|_C \le \|A - \phi(A,\tilde{w}) \tilde{w}^T\|_C
    \end{equation*}
    and
    \begin{equation*}
        \|A - \phi(A,w) w^T\|_C \ge \|A - \phi(A,\tilde{w}) \tilde{w}^T\|_C,
    \end{equation*}
    since the triples $(A,\phi(A,w), w)$ and $(A, \phi(A, \tilde{w}), \tilde{w})$ admit two-dimensional alternances and the signs of $w$ and $\tilde{w}$ coincide. Then ${E(A,v) = E(A, \tilde{v})}$.
\end{proof}

Using \thm{mainth} we can prove that it is possible to compute the distance from a matrix $A$ to the set of all rank-1 matrices in Chebyshev norm by a finite number of runs of the alternating minimization method. Namely, we prove the following

\begin{theorem}\label{theorem:global_min}
    Let $A \in \mathbb R^{m \times n}$ preserve Chebyshev systems and let $G_A$ be the sign transition graph for $A$. Let $L \subset \{-1, 1\}^n$ be the set of all loop vertices of $G_A$. Then 
    $$\inf\{ \|A - u v^T\|_C :u \in \mathbb R^m,v \in \mathbb R^n\} = \min_{t \in L} E(A,t).$$
    \begin{proof}
        Let 
        \begin{equation*}
            d = \inf\{\|A - u v^T\|_C:u \in \mathbb R^m,v \in \mathbb R^n\}.
        \end{equation*}
        From the definition of $E(A,t)$ it is clear that $E(A,t) \ge d$ for all $t \in \mathrm{Ch}_n$. Thus, $\min_{t \in L} E(A,t) \ge d$. To prove the converse inequality at first note that
        \begin{equation*}
            d = \inf\{\|A - u v^T\|_C:u \in \mathbb R^m,v \in \mathrm{Ch}_n\},
        \end{equation*}
        because Chebyshev vectors are dense in $\mathbb R^n$. Moreover, for $v \in \mathrm{Ch}_n$ we have
        \begin{equation*}
            E(A,v) \le \|A - \phi(A, v) v^T\|_C = \inf \{\|A - u v^T\|_C: u \in \mathbb R^m\}.
        \end{equation*}
        Therefore,
        \begin{equation*}
            d \ge \inf\{E(A,v):v \in \mathrm{Ch}_n\}.
        \end{equation*}
        \thm{mainth} implies that
        \begin{equation*}
        \inf\{E(A,v):v \in \mathrm{Ch}_n\} = \min\{E(A,t): t \in \{-1,1\}^n\}.    
        \end{equation*}
        Finally, since
        \begin{equation*}
            E(A,v) = E(A,\psi(A,\phi(A,v))) = E(A,\mathcal S(\psi(A,\phi(A,v))))
        \end{equation*}
        for all $v \in \mathrm{Ch}_n$ and the iterations of the map $S(\psi(A,\phi(A,v)))$ eventually map all $t \in \{-1, 1\}^n$ into $L$ by \lem{G_V_basic}~\ref{basicII}, we get that
        \begin{equation*}
            \inf\{E(A,v):v \in \mathrm{Ch}_n\} = \min_{t \in L} E(A,t).
        \end{equation*}
    \end{proof}
\end{theorem}

\begin{remark}
1. The formula for the distance
\begin{equation*}
    d = \inf\{ \|A - u v^T\|_C :u \in \mathbb R^m,v \in \mathbb R^n\}
\end{equation*}
given in \thm{global_min} can be done slightly more effective in view of the fact that $E(A,v) = E(A,-v)$ (see \lem{basic_prop_E}~\ref{basic_prop_E_i}). So, in the notation of \thm{global_min}, the number $d$ can be found by taking $\min_{t \in \tilde{L}} E(A,t)$, where ${\tilde{L} \subset L}$ is a subset such that for all $t \in L$ either $t \in \tilde{L}$, or $-t \in \tilde{L}$. Thus, the number of runs of alternating minimization method can be halved in comparison to the number of elements in $L$, since clearly $L = -L$ by \lem{basic_prop}~\ref{basic_prop_iii}.

2. For matrices $A \in \mathrm{PC}_{m,n}$ that belong to the same component of $\mathrm{PC}_{m,n}$ with some rank-1 matrix the results are even more satisfactory. Namely,
\begin{equation*}
    \inf\{ \|A - u v^T\|_C :u \in \mathbb R^m,v \in \mathbb R^n\} = E(A,v)
\end{equation*}
for arbitrary Chebyshev vector $v$. Indeed, this equality follows from \thm{global_min} and previous remark, since there are only two loop vertices $t_1$ and $t_2$ in the Graph $G_A$ and, obviously, $t_1 = - t_2$.

3. The basic corollary of \lem{existence_alternance}, that states an existence of an alternance for a triple $(A,\phi(A,v), v)$, where $v$ is a limit point of the alternating minimization method is contained in \cite{daugavet1971uniform}. We included the proofs here, since our proof is significantly shorter (and, hopefully, easier to understand) in comparison with the one presented there.

4. Our convergence results are concerned only with the convergence of the sequence $\|A - u^{(k)} (v^{(k)})^T\|_C$, where vectors $u^{(k)}, v^{(k)}$ are constructed by the alternating minimization method. As for the convergence of the sequences $u^{(k)}$ and $v^{(k)}$, we are not able to present any satisfactory results (even the boundedness). However, in our numerical experiments these sequences always converge. Similar properties are shared by the well-known ALS method (see, e.g. \cite{mohlenkamp2013musings}).

5. The properties of the sequences $u^{(k)}$ and $v^{(k)}$ can be analyzed further in some simple cases. For example, it is possible to prove that these sequences converge for $2 \times 2$ matrices that preserve Chebyshev systems. Also we are able to prove that these sequences converge for the identity matrix. However, in the general case we are not able to prove the convergence, even if the starting point $v^{(0)} = v \in \mathrm{Ch}_n$ is chosen to give an alternance for the triple $(A, \phi(A,v), v)$. However, in the foregoing case it is clear that the sequences constructed by the alternating minimization method are bounded.
 
\end{remark}

\section{Numerical results}
\label{sec:numres}
The constructed theory makes it possible to propose an algorithm for finding optimal rank-1 approximations in the Chebyshev norm. For this section we fix a matrix $A \in \mathrm{PC}_{m,n}$. In order to compute the optimal approximation, due to \thm{global_min}, it suffices to run the method of alternating minimization from all loop vertices of the graph $G_A$. However, actually it is redundant. Indeed, consider a limit point $w$ of the sequence constructed by the method of alternating minimization started from a vector $v^{(0)}$, and let $s_1, s_2, \dots, s_k$ denote the columns on which a two-dimensional alternance of $(A,\phi(A,w), w)$ is formed. Then it is not necessary to run the method from vectors $v$ such that the stabilized signs of the vectors obtained by the alternating minimization method from $v$ coincide with signs of $w$ at positions $s_1, s_2, \dots, s_k$, since $E(A,v) \ge E(A,v^{(0)})$ for such vectors (see \thm{daugavet_main}).

We recall that $F$ denotes the set of all loop vertices of the graph $G^{sd}_A$ (see \thm{sign_thm}). For $j \in \{1, \dots, n\}$ we define $L_A(j) \in F$ as the (unique) loop vertex from which the vertex $j$ can be reached in the sense of \lem{G_sd_lemma}~\ref{G_sd_iv}. From the definition of $G^{sd}_A$ it is easy to propose an algorithm to find $L_A(j)$ in terms of the mappings $\mathfrak i$ and $\mathfrak j$, defined in \sect{mainres}. Indeed, to find $L_A(j)$ calculate the sequence $j_0, j_1, j_2, \dots$ by the rule $j_k = \mathfrak j(\mathfrak i(j_{k-1}))$ with $j_0 = j$ until $j_{k+1} = j_k$ for some $k$. In this case $L_A(j) = j_k$. It is clear that the stabilization occurs for $k < n$.

Finally, we introduce a way to store the information about previous runs of the alternating minimization algorithm in order to minimize further calculations. Let $\mathscr B$ denote the set of boolean functions in the variables $x_j$, where $j \in F$, so an element $B \in \mathscr B$ is a function $B:\prod_{j \in F} \{0, 1\} \to \{0, 1\}$. With such a function we associate a subset $P(B)$ of $\prod_{j \in F} \{-1, 1\}$ by the rule 
\begin{equation*}
    P(B) = \{x \in \prod_{j \in F} \{-1, 1\}: B(b(x)) = 1\},
\end{equation*}
where $b(x)_j = 1$, if $x_j = 1$ and $b(x)_j = 0$, if $x_j = -1$. We shall store the information about previous runs in the form of a disjunctive normal form (DNF for short) $B$ defined in a way, such that starting vectors $v$ satisfying $\mathcal S(v)|_F \in P(B)$ are less optimal than one of the previous runs.

Now we describe the idea of an algorithm that computes the distance from $A$ to the set of all rank-1 matrices in Chebyshev norm. From \thm{global_min} we deduce that to compute the distance it suffices to run the alternating minimization method from a set of vertices $V \subset \{-1,1\}^n$ that intersects each component of $G_A$. That is, \thm{sign_thm}~\ref{thm_i} suggests that it is not necessary to compute the components of loop vertices outside $F$, so we can set the components $v_j$ of the initial points arbitrarily for $j \notin F$. Moreover, after we compute a limit point $w$ and the columns $s_1, s_2, \dots, s_k$ on which the alternance is formed, we store the information about the alternance in order to minimize further calculations. In order to do so we introduce a DNF $B$ which is initialized by $0$ at the beginning. After the computation of $w$ and $s_1, \dots, s_k$ we update $B$ in the following way:
\begin{equation}\label{eq:DNF_update}
B \to B \lor x_{L_A(s_1)}^{\sign(w_{L_A(s_1)})}x_{L_A(s_2)}^{\sign(w_{L_A(s_2)})}\dots x_{L_A(s_k)}^{\sign(w_{L_A(s_k)})}.\end{equation}
Here, as usual, we use the notation $x^\nu$ in the sense
\begin{equation*}
        x^\nu = \begin{cases}
            x, & \nu = 1 \\
            1-x, & \nu=-1
        \end{cases}
\end{equation*}
Let $v$ be a Chebyshev vector such that $\mathcal S(v)|_F$ belongs to $P(B')$, where $B'$ denotes the conjunction from \eq{DNF_update}. From \lem{G_V_components}~\ref{comp_iii} it is clear that the loop vertex of $G_A$ that lies in the same component with $\mathcal S(v)$ coincides with $\mathcal S(w)$ on $\{s_1, \dots, s_k\}$ and, as was noted in the beginning of this section, $E(A,v) \ge E(A,w)$. Thus, it is not necessary to perform any computations with $v$. Therefore, the next starting point $v$ should be taken in a way such that $\mathcal S(v)|_F$ does not belong to $P(B)$. Thus, repeating the foregoing procedure while $P(B) \ne \prod_{j \in F}\{-1, 1\}$ suffices to find the optimal rank-1 approximation. Below we outline the main steps of the obtained algorithm.
\begin{enumerate}
    \item Find the set of positions in the matrix $A$ such that the element in this position is the maximum absolute value in its row and its column. Denote the set of corresponding indices of columns by $F$.
    \item Initialize a DNF $B$ with variables $x_j$, $j \in F$, as zero.
    \item Find a vector $v \in \{-1, 1\}^n$ such that $v|_F$ and $(-v)|_F$ do not belong to $P(B)$ (in particular, values $v_j$ for $j \notin F$ can be defined arbitrarily). If such a vector is not possible to find, then terminate.
    \item Perform the alternating minimization method with the starting point $v$ and find the limit point $w$ of the sequence $v^{(k)}/\|v^{(k)}\|_\infty$. Also find the columns $s_1, \dots, s_k$, where a two-dimensional alternance of the triple $(A,\phi(A,w), w)$ is positioned.
    \item Update the DNF $B$ by the formula \eq{DNF_update} and return to the step 3.
\end{enumerate}
The reasoning above ensures that $\inf \{\|A - u v^T\|_C :u \in \mathbb R^m,v \in \mathbb R^n\}$ coincides with the $\min E(A,v)$, where the minimum is taken over all $v$ that were considered in the step 3 of the algorithm.
\begin{figure}[ht!]
\centering
\begin{subfigure}{.5\textwidth}
  \centering
  \includegraphics[width=.98\linewidth]{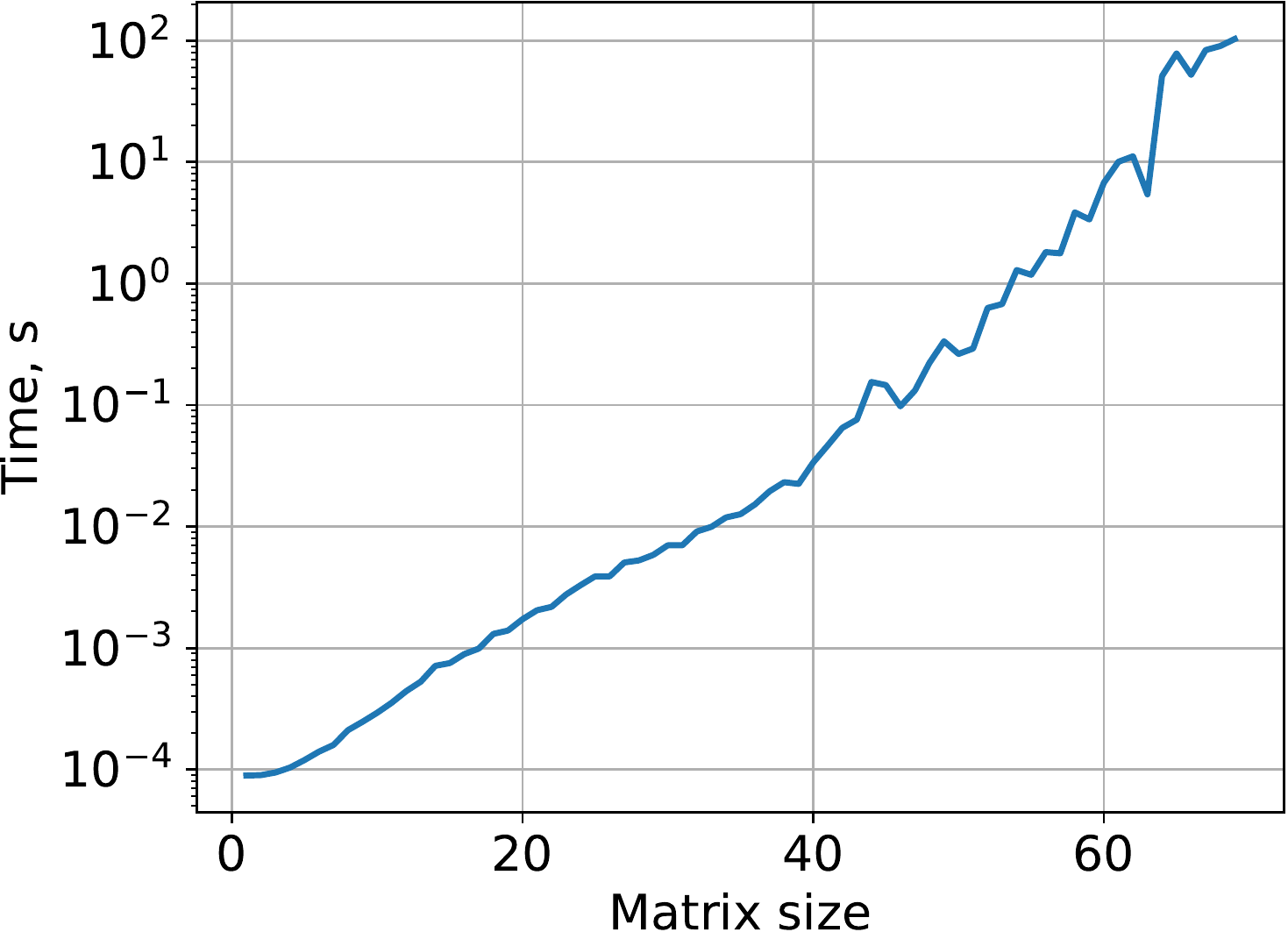}
\end{subfigure}
\caption{The average running time of the algorithm depending on the size of the matrix. The matrices are random from a normal distribution, all results are averaged over 200 runs.}
\label{fig:optimal_time}
\end{figure}

\begin{figure}[ht!]
\centering
\begin{subfigure}{.5\textwidth}
  \centering
  \includegraphics[width=.98\linewidth]{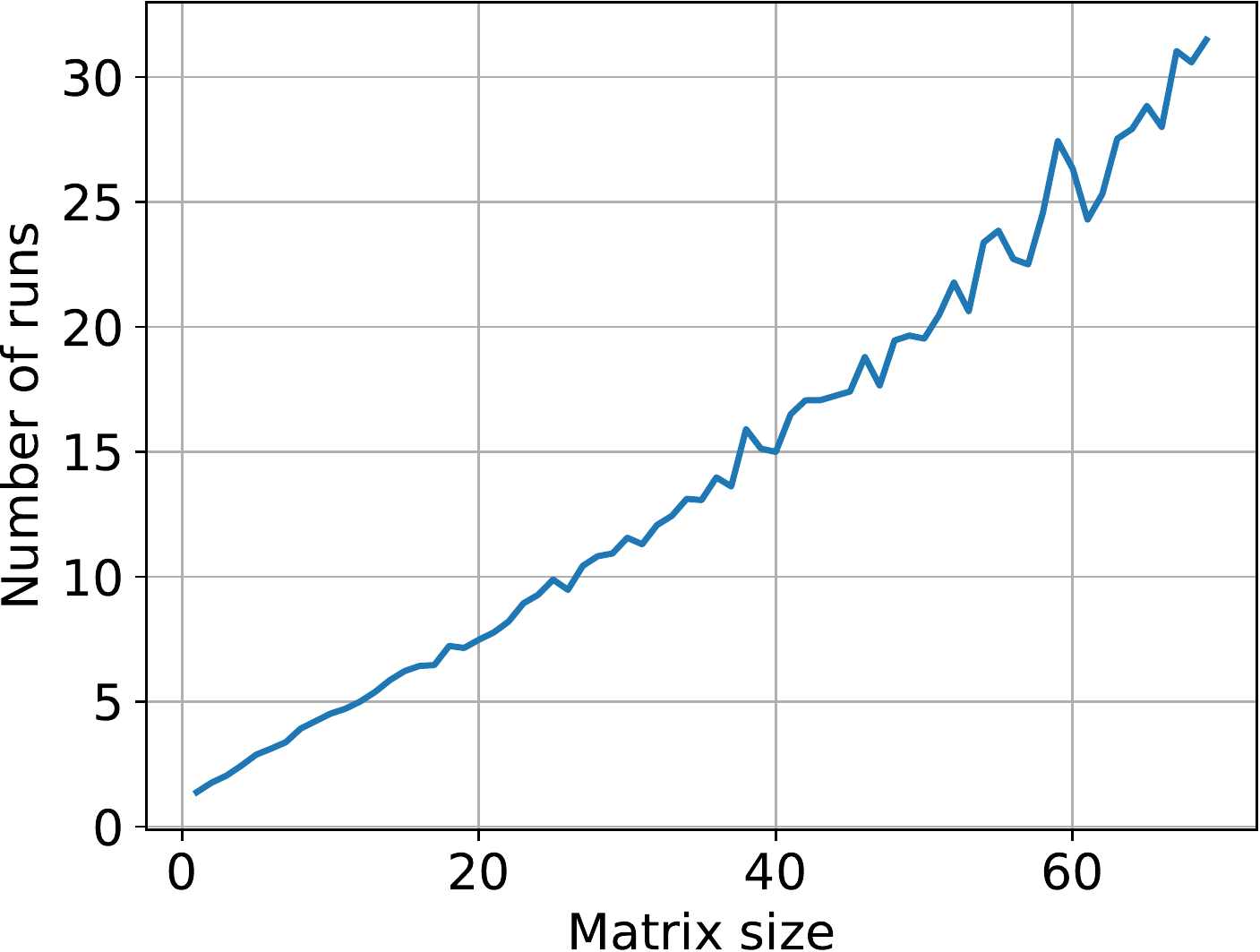}
\end{subfigure}
\caption{The average number of runs of alternating minimization depending on the size of the matrix. The matrices are random from a normal distribution, all results are averaged over 200 runs.}
\label{fig:optimal_num_runs}
\end{figure}

It is noteworthy that the problem of checking a DNF for being equal to 1 is an NP-complete problem (see, e.g. \cite[Section~4]{Karp2010}), thus, we are not always able to understand that the DNF $B$ identically equals 1, even if it is. In practice, we recursively generate components of the new vector and substitute already generated components in the current DNF. If after the substitution one of the conjunctions is identically equal to 1 with respect to the remaining variables, then we stop to generate this branch of the recursion. Otherwise, we continue to generate components until all the variables have assigned value.

The described algorithm was implemented in C++. We emphasize that the sequences generated by the alternating minimization method in numerical experiments always converge and, therefore, in the step 4 of our implementation of the algorithm there is no need in normalization and extraction of a convergent subsequence. Entries of matrices in all experiments are chosen randomly and independently from the standard normal distribution. Fig.~2 shows the average running time for square matrices of various sizes. Fig.~3 shows the number of runs of the alternating minimization method in a single performance of the foregoing algorithm (i.e. how many times this algorithm will enter the fourth step). It is noteworthy that the number of runs grows linearly with the size of the problem, and the running time of the alternating minimization method is negligible compared to the total running time. Thus, a more efficient organization of the combinatorial optimization process could give significantly better results, but we leave this for further investigation.

\section{Conclusion}
\label{sec:conclusion}
In the paper we proved that the result of the alternating minimization method depends only on the signs of the components of the starting point, and studied how the signs of the approximation vectors change when the alternating minimization method applied. As a result, a method is constructed that is capable of constructing optimal Chebyshev approximations of rank 1 for moderate matrices.

We note, that in numerical experiments the sequences, generated by the alternating minimization algorithm, are always convergent. However, we are not able to prove (or disprove) this fact.
% and, therefore, in this paper we often pass to convergent subsequences, instead of analyzing the initial sequences.
Also it is noteworthy that similar structures can be observed in behaviour of the same algorithm for rank-r approximations. For now we are not able to provide a satisfactory analysis in the general case ($r > 1$).


% \appendix

% \section{Sample Appendix Section}
% \label{sec:sample:appendix}
% Lorem ipsum dolor sit amet, consectetur adipiscing elit, sed do eiusmod tempor section \ref{sec:sample1} incididunt ut labore et dolore magna aliqua. Ut enim ad minim veniam, quis nostrud exercitation ullamco laboris nisi ut aliquip ex ea commodo consequat. Duis aute irure dolor in reprehenderit in voluptate velit esse cillum dolore eu fugiat nulla pariatur. Excepteur sint occaecat cupidatat non proident, sunt in culpa qui officia deserunt mollit anim id est laborum.

 \bibliographystyle{elsarticle-num} 
 \bibliography{cas-refs}

\begin{thebibliography}{10}
\expandafter\ifx\csname url\endcsname\relax
  \def\url#1{\texttt{#1}}\fi
\expandafter\ifx\csname urlprefix\endcsname\relax\def\urlprefix{URL }\fi
\expandafter\ifx\csname href\endcsname\relax
  \def\href#1#2{#2} \def\path#1{#1}\fi

\bibitem{bebendorf2008means}
M.~Bebendorf, A means to efficiently solve elliptic boundary value problems,
  Hierarchical Matrices. LNCS 63 (2008) 49--98.

\bibitem{son2014data}
S.~W. Son, Z.~Chen, W.~Hendrix, A.~Agrawal, W.-k. Liao, A.~Choudhary, Data
  compression for the exascale computing era-survey, Supercomputing frontiers
  and innovations 1~(2) (2014) 76--88.

\bibitem{he2016fast}
X.~He, H.~Zhang, M.-Y. Kan, T.-S. Chua, Fast matrix factorization for online
  recommendation with implicit feedback, in: Proceedings of the 39th
  International ACM SIGIR conference on Research and Development in Information
  Retrieval, 2016, pp. 549--558.

\bibitem{yang2018oboe}
C.~Yang, Y.~Akimoto, D.~Kim, M.~Udell, Oboe: Collaborative filtering for automl
  initialization, arXiv preprint arXiv:1808.03233 (2018).

\bibitem{udell2019big}
M.~Udell, A.~Townsend, Why are big data matrices approximately low rank?, SIAM
  Journal on Mathematics of Data Science 1~(1) (2019) 144--160.

\bibitem{zamarashkin2022best}
N.~Zamarashkin, S.~Morozov, E.~Tyrtyshnikov, On the best approximation
  algorithm by low-rank matrices in {C}hebyshev’s norm, Computational
  Mathematics and Mathematical Physics 62~(5) (2022) 701--718.

\bibitem{daugavet1971uniform}
V.~Daugavet, Uniform approximation of a function of two variables, tabulated as
  the product of functions of a single variable, USSR Computational Mathematics
  and Mathematical Physics 11~(2) (1971) 1--16.

\bibitem{gillis2019low}
N.~Gillis, Y.~Shitov, Low-rank matrix approximation in the infinity norm,
  Linear Algebra and its Applications 581 (2019) 367--382.

\bibitem{mohlenkamp2013musings}
M.~J. Mohlenkamp, Musings on multilinear fitting, Linear Algebra and its
  Applications 438~(2) (2013) 834--852.

\bibitem{Karp2010}
R.~M. Karp, Reducibility Among Combinatorial Problems, Springer Berlin
  Heidelberg, Berlin, Heidelberg, 2010, pp. 219--241.

\end{thebibliography}





% \end{thebibliography}
\end{document}